\documentclass[a4paper,reqno,11pt]{amsart}
\usepackage{amsmath,amssymb,amsthm}
\usepackage[a4paper,margin=2.8cm]{geometry}
\usepackage{enumitem}
\usepackage{needspace}
\usepackage[colorlinks=true,linkcolor=blue,citecolor=blue,urlcolor=blue]{hyperref}
\usepackage[all]{xy}

\newtheorem{theorem}{Theorem}[section]
\newtheorem{proposition}[theorem]{Proposition}
\newtheorem{lemma}[theorem]{Lemma}
\newtheorem{corollary}[theorem]{Corollary}
\theoremstyle{definition}
\newtheorem{definition}[theorem]{Definition}
\newtheorem{example}[theorem]{Example}
\newtheorem{question}[theorem]{Question}
\newtheorem{setup}[theorem]{Setup}
\theoremstyle{remark}
\newtheorem{remark}[theorem]{Remark}

\numberwithin{equation}{section}
\newcommand{\Hom}{\operatorname{Hom}}
\newcommand{\Ext}{\operatorname{Ext}}
\newcommand{\End}{\operatorname{End}}
\newcommand{\Tor}{\operatorname{Tor}}
\newcommand{\add}{\operatorname{add}}
\renewcommand{\mod}{\operatorname{mod}}
\newcommand{\rej}{\operatorname{rej}}
\newcommand{\tr}{\operatorname{tr}}
\newcommand{\Top}{\operatorname{top}}
\newcommand{\rad}{\operatorname{rad}}
\newcommand{\Ker}{\operatorname{Ker}}
\newcommand{\Cok}{\operatorname{Cok}}
\newcommand{\Img}{\operatorname{Im}}
\newcommand{\pd}{\operatorname{pd}}
\newcommand{\gldim}{\operatorname{gldim}}
\newcommand{\cores}{\operatorname{cores}}
\newcommand{\supp}{\operatorname{supp}}
\newcommand{\sss}{\operatorname{\mathsf{sss}}}
\newcommand{\qhs}{\operatorname{\mathsf{qhs}}}
\newcommand{\tilt}{\operatorname{\mathsf{tilt}}}
\newcommand{\Ttilt}{\operatorname{\mathsf{Ttilt}}}
\newcommand{\UTtilt}{\operatorname{\mathsf{UTtilt}}}
\newcommand{\F}{\mathcal{F}}
\newcommand{\Y}{\mathcal{Y}}
\newcommand{\PP}{\mathcal{P}}
\newcommand{\QQ}{\mathcal{Q}}
\newcommand{\DD}{\mathbb{D}}
\newcommand{\op}{\mathrm{op}}

\begin{document}

\title{Standardly stratified structures from tilting modules}
\author{Mayu Tsukamoto}
\address{M.~Tsukamoto: Graduate School of Mathematics, Nagoya University, Furo-cho, Chikusa-ku, Nagoya, Aichi 464-8602, Japan}
\email{tsukamoto@math.nagoya-u.ac.jp}

\begin{abstract}
Characteristic tilting modules play a central role in the theory of standardly stratified algebras, for instance in Ringel duality.
They are constructed from the standard modules, which are determined by a partial order on the simple modules.
In this paper, we characterize characteristic tilting modules intrinsically. 
We show that a tilting module is a characteristic tilting module if and only if the matrix of composition multiplicities of its indecomposable direct summands is triangular with respect to suitable orderings of the summands and of the simple modules.
In this case, the standard modules and all the corresponding partial orders are recovered from the tilting module.
As an application, we prove that the tilting correspondence of Geiss, Leclerc and Schr\"oer between a GLS path algebra and its associated hereditary algebra restricts to a bijection between characteristic tilting modules.
In particular, we obtain GLS path algebras that are not quasi-hereditary with respect to any partial order but for which every tilting module is a characteristic tilting module.
\end{abstract}

\maketitle

\section{Introduction}

Quasi-hereditary algebras were introduced to study highest weight categories arising in the representation theory of semisimple Lie algebras and algebraic groups \cite{CPS1}. 
Basic examples arise from blocks of the Bernstein--Gelfand--Gelfand
category $\mathcal O$ and from Schur algebras.
Standardly stratified algebras were introduced as generalizations of
quasi-hereditary algebras \cite{CPS2,D, ADL1}. 
They also arise in Lie theory through the generalizations of category
$\mathcal O$ studied by Futorny, K\"onig and Mazorchuk, including
$\mathcal S$-subcategories whose associated algebras need not be
quasi-hereditary \cite{FKM1,FKM2,MPW}.
The structural and homological properties of standardly stratified algebras have been studied systematically by \'Agoston, Dlab and Luk\'acs \cite{ADL1,ADL2}.
Building on Dlab and Ringel's standardization method \cite{DR},
various authors have studied the construction of standardly stratified algebras via stratifying systems \cite{CPS2, ES,MS,MT,AT}.
Other developments include approaches using Borelic pairs and exact
Borel subalgebras \cite{CZ,CK}.

Being standardly stratified is not a property of an algebra alone.
A standardly stratified algebra is defined as a pair $(A,\unlhd)$ of a finite dimensional algebra $A$ and a partial order $\unlhd$ on the indices of the simple $A$-modules $S(1),\dots,S(n)$, satisfying certain homological conditions. 
The partial order determines a family of modules, called the \emph{standard modules}, which play a central role in the theory of standardly stratified algebras.
In general, an algebra may be standardly stratified with respect to several partial orders.
We regard two such orders as equivalent if they determine the same standard modules, and call an equivalence class a \emph{standardly stratified structure}.
Standardly stratified and quasi-hereditary structures have been
studied from various perspectives, including characterizations
of algebras that are standardly stratified with respect to
every total order \cite{HM,Li}, uniqueness of structures 
\cite{Cou,MPW}, their poset structures and combinatorics \cite{FKR},
and connections with tilting theory \cite{ACKT}.

Tilting theory is one of the central tools in the representation theory of finite dimensional algebras. Tilting modules induce equivalences of derived categories, and tilting objects also appear in algebraic geometry. Moreover, mutation and a natural partial order provide a combinatorial approach to the study of tilting modules \cite{RS,HU}. A fundamental connection with quasi-hereditary algebras was established by Ringel \cite{R}. Every quasi-hereditary structure determines a distinguished tilting module, called the \emph{characteristic tilting module}, whose endomorphism algebra is again quasi-hereditary. This construction, known as Ringel duality, was extended to standardly stratified algebras \cite{AHLU,Xi, Fr2}.
Thus a partial order determines the standard modules, and these
determine the structure and its characteristic tilting module.
On the other hand, the notion of a tilting module does not
refer to any partial order.
This leads to the following inverse problem.

\begin{question}\label{quest:intro}
\begin{enumerate}[label=\textup{(\arabic*)}]
\item Given a tilting module, can one determine whether it is a characteristic tilting
module from properties of the module itself, without specifying a partial order in advance?
\item If it is, can one recover the standard modules and all the partial orders of the
corresponding structure from the tilting module?
\end{enumerate}
\end{question}

For quasi-hereditary algebras over an algebraically closed field, Question~\ref{quest:intro}(1) was
answered in \cite{ACKT}. 
In this paper, we answer both questions for standardly stratified algebras in terms of the supports of the indecomposable direct summands of a tilting module. 
We call a basic tilting $A$-module $T$ \emph{triangular} if its indecomposable direct summands can be labelled as $T(1),\dots,T(n)$ so that, for some partial order $\unlhd$,  the support $\supp T(i)$ of each $T(i)$ is contained in $\{j\mid j\unlhd i\}$.
Our main result answers Question~\ref{quest:intro}.
 
\begin{theorem}[Theorem~\ref{thm:main}, Propositions~\ref{cor:canonical}
and~\ref{prop:essential}]\label{thmA}
Let $A$ be a basic finite dimensional algebra over a field $k$, and let $T$ be a basic
tilting $A$-module.
Then $T$ is the characteristic tilting module of a standardly stratified structure on
$A$ if and only if $T$ is triangular.
In this case, the structure is uniquely determined by $T$, and the labelling $T(1),\dots,T(n)$ satisfying the triangular condition is unique.
 Moreover, the following statements hold.
\begin{enumerate}[label=\textup{(\arabic*)}]
\item The relations $j\preceq_T i$ for all $j\in\supp T(i)$ generate a partial order
$\preceq_T$, and the partial orders in the structure are precisely those containing
$\preceq_T$. In particular, $\preceq_T$ is the minimum partial order in the structure.
\item The standard module $\Delta(i)$ in the structure is isomorphic to the
trace of the projective cover of $S(i)$ in $T(i)$.
\end{enumerate}
\end{theorem}

Using Theorem~\ref{thmA}, we show an isomorphism between
the poset of triangular tilting $A$-modules and that of
standardly stratified structures on $A$ (Theorem~\ref{thm:bijection}).
Restricted to quasi-hereditary structures, it extends \cite[Theorem~3.22]{ACKT} to arbitrary fields.
Moreover, $\preceq_T$ coincides with the essential order of \cite{MPW}, which is the minimal adapted order of \cite{FKR} in the quasi-hereditary case (Remark~\ref{rem:least-order}).
Theorem~\ref{thmA} also yields a recursive criterion for a given tilting module to be triangular (Corollary~\ref{cor:algorithm}), and shows that triangularity is preserved under Ringel duality (Corollary~\ref{cor:ringel}).

As an application, we study standardly stratified structures on a class of tensor algebras including the GLS path algebras
introduced by Geiss, Leclerc and Schr\"oer \cite{GLS-found} in connection with
symmetrizable Kac--Moody Lie algebras.
Such an algebra admits no quasi-hereditary structure unless it is hereditary, but it is
standardly stratified with respect to every total order.
Moreover, its standardly stratified structures correspond to the quasi-hereditary
structures on the path algebra of its associated acyclic quiver (Theorem~\ref{thm:tensor} and
Corollary~\ref{cor:GLS}).
For a GLS path algebra $H$, Geiss, Leclerc and Schr\"oer gave a bijection between tilting
$H$-modules and tilting modules over a corresponding hereditary algebra
\cite{GLS-rigid}.
Since this bijection preserves supports, Theorem~\ref{thmA} shows that it restricts to a
bijection between triangular tilting modules (Theorem~\ref{thm:GLS-compatibility}).
Combining this with \cite[Theorem~4.7]{FKR}, we show that every tilting module over the GLS path algebra
$k[x]/(x^m)\otimes_kk\vec{\mathbb{A}}_n$ is triangular, although this algebra has no
quasi-hereditary structure for $m\ge2$ (Corollary~\ref{cor:GLS-linear-tilting}).

\subsection*{Acknowledgment}
The author is supported by JSPS Grant-in-Aid for Early-Career Scientists JP23K12959 and by the RDP Program Aiming at Maximizing the Career Potential of Female Researchers, Nagoya University (MEXT's Initiative for Realizing Diversity in the Research Environment, Leadership training type for women). 

\subsection*{Use of artificial intelligence}
The author used ChatGPT (OpenAI) and Claude (Anthropic) to assist in improving some of the mathematical arguments and the exposition, and in editing the English. 
The author takes full responsibility for the mathematical content and the final text.

\section{Preliminaries}\label{sec:prelim}

\subsection{Notation and convention}
Throughout this paper, $k$ is a field and $A$ is a basic finite dimensional $k$-algebra. 
We write $\DD=\Hom_{k}(-,k)$ for the $k$-linear dual. 
We assume that all modules are finitely generated right modules, and all subcategories are full, unless otherwise specified. 

Denote by $\mod A$ the category of $A$-modules. 
For $M\in\mod A$, $\add M$ denotes the subcategory of $\mod A$ consisting of direct summands of finite direct sums of copies of $M$. 
For a class $\mathcal{X}$ of $A$-modules, we say that an $A$-module $M$ admits an $\mathcal{X}$-filtration if there exists a filtration $0=M_0\subseteq M_1\subseteq\dots\subseteq M_l=M$ such that $M_j/M_{j-1}$ is isomorphic to a module in $\mathcal{X}$ for all $1\le j\le l$. 
Let $\F(\mathcal{X})$ be the subcategory of $\mod A$ consisting of $A$-modules that admit an $\mathcal{X}$-filtration. 
 We also write $\F(M)$ for $\F(\{M\})$ and put $\F(\emptyset)=\{0\}$. 
For a finite index set $I$, a family $\Theta=(\Theta(i))_{i\in I}$ of $A$-modules, and a partial order $\unlhd$ on $I$, we write $\F(\Theta(\unlhd i))=\F(\{\Theta(j)\mid j\unlhd i\})$ and $\F(\Theta(\lhd i))=\F(\{\Theta(j)\mid j\lhd i\})$, where $j\lhd i$ means $j\unlhd i$ and $j\ne i$. We define $\F(\Theta(\unrhd i))$ and $\F(\Theta(\rhd i))$ similarly. 
For a class $\mathcal{X}$ of $A$-modules, we define
\[
\begin{aligned}
\mathcal{X}^{\perp_1}&=\{Y\in\mod A\mid\Ext_A^1(X,Y)=0\text{ for all }X\in\mathcal{X}\},\\
\mathcal{X}^{\perp}&=\{Y\in\mod A\mid\Ext_A^m(X,Y)=0\text{ for all }X\in\mathcal{X},\ m\ge1\}.
\end{aligned}
\]
Similarly, we define subcategories ${}^{\perp_1}\mathcal{X}$ and ${}^{\perp}\mathcal{X}$. 
For a singleton, we write, for example, $T^\perp$ instead of $\{T\}^\perp$. We also put $\Y(\mathcal{X})=\F(\mathcal{X})^{\perp_1}$. 

We fix a complete set $\{e_1,\dots,e_n\}$ of primitive orthogonal idempotents of $A$ and put $\Lambda=\{1,\dots,n\}$. 
For each $i\in\Lambda$, let $P(i)=e_iA$, $I(i)=\DD(Ae_i)$ be the associated indecomposable projective and injective module, respectively, and let $S(i)$ be the simple top of $P(i)$. 

We write $[M:S(i)]_A$ for the composition multiplicity of $S(i)$ in an $A$-module $M$, and define the \emph{support} of $M$ by
\[
\supp_A M
=\{i\in\Lambda\mid Me_i\ne 0\}
=\{i\in\Lambda\mid [M:S(i)]_A\ne 0\}.
\]
We omit the subscript $A$ whenever no confusion can arise.
For modules $M$ and $N$, the \emph{reject} of $N$ in $M$ and the \emph{trace} of $N$ in $M$ are
\[
\rej_N(M)=\bigcap_{h\in\Hom_A(M,N)}\Ker h,\qquad \tr_N(M)=\sum_{h\in\Hom_A(N,M)}\Img h .
\]

A total order $\le$ on $\Lambda$ is called a \emph{linear extension} of a partial order $\unlhd$ if $j\unlhd i$ implies $j\le i$. 
If $\sqsubseteq$ and $\unlhd$ are partial orders, we say that $\unlhd$ \emph{contains} or \emph{refines} $\sqsubseteq$ when ${\sqsubseteq}\subseteq{\unlhd}$, that is, $i\sqsubseteq j$ implies $i\unlhd j$. 

\subsection{Standardizable sets}
We recall the standardizable sets arising in the Dlab--Ringel standardization method (\cite[Theorem~2]{DR}, \cite{ES}). 
In this subsection, we collect the definitions and basic properties of standardizable sets for later use.

\begin{definition}\label{def:strat}
Let $\Theta=(\Theta(1),\dots,\Theta(t))$ be an ordered family of $A$-modules, and consider the following conditions.
\begin{enumerate}[label=\textup{(S\arabic*)}]
\item Each $\Theta(i)$ is indecomposable.
\item $\Hom_A(\Theta(i),\Theta(j))=0$ whenever $i>j$.
\item $\Ext^1_A(\Theta(i),\Theta(j))=0$ whenever $i\ge j$.
\end{enumerate}
\begin{enumerate}[label=\textup{(\arabic*)}]
\item $\Theta$ is called a \emph{standardizable set} if it satisfies \textup{(S1)}, \textup{(S2)} and \textup{(S3)}.
\item A pair $(\Theta,Y)$ is called a \emph{costratifying system} if $\Theta$ satisfies \textup{(S1)} and \textup{(S2)}, and $Y=(Y(1),\dots,Y(t))$ is a family of indecomposable modules such that, for each $i$, $Y(i)\in\Y(\Theta)$ and there is an exact sequence $0\to\Theta(i)\to Y(i)\to Z(i)\to0$ with $Z(i)\in\F(\Theta(<i))$.
\end{enumerate}
\end{definition}

We identify the family $Y$ with the module $\bigoplus_iY(i)$.

\begin{proposition}[{\cite[Sections~1--3]{R}, \cite[Section~1]{ES}}]\label{prop:AT}
Let $\Theta=(\Theta(1),\dots,\Theta(t))$ be an ordered family of indecomposable $A$-modules.
Then the following statements hold. 
\begin{enumerate}[label=\textup{(\arabic*)}]
\item If $\Theta$ satisfies \textup{(S3)}, then every $M\in\F(\Theta)$ has a filtration $0=M_{t+1}\subseteq M_t\subseteq\dots\subseteq M_1=M$ with $M_i/M_{i+1}\in\add \Theta(i)$ for all $i$. 
 If $\Theta$ is a standardizable set, then $\F(\Theta)$ is closed under direct summands.
\item If $(\Theta,Y)$ is a costratifying system, then $\Theta$ satisfies \textup{(S3)},
\[
\add Y=\F(\Theta)\cap\Y(\Theta),
\]
and $Y$ is an injective cogenerator of $\F(\Theta)$. Thus every $M\in\F(\Theta)$ admits an exact sequence $0\to M\to Y'\to M'\to0$ with $Y'\in\add Y$ and $M'\in\F(\Theta)$.
\end{enumerate}
\end{proposition}

If a standardizable set consists of modules with pairwise distinct simple tops, then the following holds.

\begin{lemma}\label{lem:top}
Let $\Theta=(\Theta(1),\dots,\Theta(n))$ be a standardizable set with $\Top\Theta(j)\cong S(j)$ for every $j$, and let $M\in\F(\Theta)$ with $\Top M\cong S(i)$. Then there is an exact sequence $0\to M'\to M\to\Theta(i)\to0$ with $M'\in\F(\Theta(>i))$.
\end{lemma}

\begin{proof}
Let $M\in\F(\Theta)$ with $\Top M\cong S(i)$. 
By Proposition~\ref{prop:AT}(1), $M$ has a filtration $0=M_{n+1}\subseteq M_n\subseteq\dots\subseteq M_1=M$ with $M_j/M_{j+1}\in\add\Theta(j)$. Let $l$ be minimal with $M_l\ne M_{l+1}$, so that $M_l=M$. Then $M/M_{l+1}$ is a nonzero module in $\add\Theta(l)$ whose top is a nonzero factor module of $\Top M\cong S(i)$. Since $\Top\Theta(l)\cong S(l)$, this forces $M/M_{l+1}\cong\Theta(l)$ and $l=i$. 
Taking $M'=M_{i+1}$, we have $M'\in\F(\Theta(>i))$, and the assertion follows.
\end{proof}

\subsection{Standardly stratified algebras}\label{subsec:standard}
We recall the definitions of standard and proper costandard modules.
Let $\unlhd$ be a partial order on $\Lambda$. 
The \emph{standard module} $\Delta(i)$ is the maximal factor module of $P(i)$ whose support is contained in $\{j\mid j\unlhd i\}$. 
The \emph{proper costandard module} $\bar\nabla(i)$ is the maximal submodule of $I(i)$ whose support is contained in $\{j\mid j\unlhd i\}$ and in which $S(i)$ occurs with multiplicity one.
We put $\Delta=\{\Delta(i)\mid i\in\Lambda\}$ and $\bar\nabla=\{\bar\nabla(i)\mid i\in\Lambda\}$.

For $i\in\Lambda$, put $f_i=\sum_{j\not\unlhd i}e_j$, and define $K(i)=e_iAf_iA$ and $D_i=e_iAe_i/K(i)e_i$.  
Then
\[
\Delta(i)=P(i)/K(i),\qquad K(i)=\tr_{\bigoplus_{j\not\unlhd i}P(j)}P(i).
\]
We put $c_i=\dim_kD_i/\dim_k\End_A(S(i))$. 
When the order has to be specified, we write $\Delta^\unlhd(i)$, $K^\unlhd(i)$ and $D_i^\unlhd$.

\begin{lemma}[{\cite[Lemmas~1.1, 1.2 and 1.6]{DR}}]\label{lem:Delta}
Let $\unlhd$ be a partial order on $\Lambda$ and $i\in\Lambda$.
Then the following statements hold. 
\begin{enumerate}[label=\textup{(\arabic*)}]
\item $\End_A(\Delta(i))\cong D_i$ is a local algebra, $\Delta(i)e_i\cong D_i$ as right $D_i$-modules, and $[\Delta(i):S(i)]=c_i$. 
Moreover, $D_i$ is a division ring if and only if $c_i=1$, equivalently $[\Delta(i):S(i)]=1$.
\item $\Ext^1_A(\Delta(i),N)=0$ for every module $N$ with $\supp N\subseteq\{j\mid j\unlhd i\}$. In particular, $\Ext^1_A(\Delta(i),\Delta(j))=0$ whenever $j\unlhd i$.
\item If $\Hom_A(\Delta(i),\Delta(j))\ne0$, then $i\unlhd j$.
\end{enumerate}
Consequently, the standard modules defined using the natural order on $\Lambda$ form a standardizable set $\Delta=(\Delta(1),\dots,\Delta(n))$ with $\Top\Delta(j)\cong S(j)$ for every $j$.
\end{lemma}

\begin{proof}
The assertions in (2) and (3) follow from   \cite[Lemma~1.1]{DR} and \cite[Lemma~1.2(b)]{DR} respectively. 
For (1), put $\bar A=A/Af_iA$. Since $\Delta(i)=e_i\bar A$, we have
\[
\End_A(\Delta(i))
\cong \End_{\bar A}(e_i\bar A)
\cong e_i\bar Ae_i
= D_i.
\]
Since $e_iAe_i$ is local, so is $D_i$. Moreover, $\Delta(i)e_i=(e_i\bar A)e_i\cong e_i\bar Ae_i=D_i$ as right $D_i$-modules. 
Since applying $(-)e_i$ to a composition series of $\Delta(i)$ gives $\dim_k \Delta(i)e_i=[\Delta(i):S(i)]\dim_k S(i)e_i$, we have
\[\dim_k D_i =\dim_k \Delta(i)e_i =  [\Delta(i):S(i)]\dim_kS(i)e_i= [\Delta(i):S(i)]\dim_k\End_A(S(i)).\]
Hence $[\Delta(i):S(i)]=c_i$. 
The remaining assertions in (1) follow from \cite[Lemma~1.6]{DR}.
\end{proof}

We define standardly stratified algebras.

\begin{definition}[\cite{CPS2,D,ADL1}]\label{def:ss}
Let $\unlhd$ be a partial order on $\Lambda$. We say that $(A,\unlhd)$ is \emph{standardly stratified} if $K(i)\in\F(\Delta(\rhd i))$ for every $i\in\Lambda$. 
It is called \emph{quasi-hereditary} if, moreover, $D_i$ is a division ring, equivalently $[\Delta(i):S(i)]=1$, for every $i$.
\end{definition}

We collect some basic results on standardly stratified algebras.

\begin{lemma}[{\cite[Section~2.7]{CZ}, \cite[Lemma 8]{Fr}}]\label{lem:refine}
Let $(A,\unlhd)$ be standardly stratified, and let $\unlhd'$ be a partial order on $\Lambda$ refining $\unlhd$. 
Then $(A,\unlhd')$ is standardly stratified and $\Delta^{\unlhd'}(i)=\Delta^{\unlhd}(i)$ for every $i\in\Lambda$.
Moreover, $K^{\unlhd'}(i)=K^{\unlhd}(i)$ and $D_i^{\unlhd'}=D_i^{\unlhd}$ for every $i\in\Lambda$. 
\end{lemma}

\begin{proof}
By \cite[Section~2.7]{CZ}, $(A,\unlhd')$ is standardly stratified and $\Delta^{\unlhd'}(i)=\Delta^{\unlhd}(i)$ for every $i\in\Lambda$.
Since the natural epimorphism $\Delta^{\unlhd'}(i)\twoheadrightarrow\Delta^\unlhd(i)$ is an isomorphism, $K^{\unlhd'}(i)=K^\unlhd(i)$. 
Thus we have $D_i^{\unlhd'}
=e_iAe_i/K^{\unlhd'}(i)e_i
=e_iAe_i/K^\unlhd(i)e_i
=D_i^\unlhd$. 
\end{proof}

\begin{lemma}[{\cite[Lemma~4.5(1)]{AT}}]\label{lem:Theta-Delta}
Let $\le$ be the natural order on $\Lambda$, and let
$\Theta=(\Theta(1),\dots,\Theta(n))$ be a standardizable set such that
$\Top\Theta(i)\cong S(i)$ for every $i$. If $A_A\in\F(\Theta)$, then $(A,\le)$ is
standardly stratified and $\Theta(i)\cong\Delta^{\le}(i)$ for every $i$.
In particular, $(A,\le)$ is standardly stratified if and only if
$A_A\in\F(\Delta^{\le})$.
\end{lemma}

\begin{proof}
Since $\F(\Theta)$ is closed under direct summands,
we have $P(i)\in\F(\Theta)$ for every $i$.
By Lemma~\ref{lem:top}, there is an exact sequence
$0\to K_i\to P(i)\to \Theta(i)\to0$ with $K_i\in\F(\Theta(>i))$.
Since $\Theta$ is standardizable and each $P(i)$ is projective, we have $\Theta(i)\cong\Delta^{\le}(i)$ for every $i$ by \cite[Lemma~4.5(1)]{AT}.
Thus the exact sequence above shows that $(A,\le)$ is standardly stratified.
\end{proof}

The last assertion of Lemma~\ref{lem:Theta-Delta} does not hold for partial orders
in general, as the following example shows.

\begin{example}\label{ex:antichain}
Let $A=k Q/I$ be the bound quiver algebra given by
\begin{align}
Q:\;\;
\xymatrix{
1\ar@<0.7mm>[r]^-{a}
&
2\ar@<0.7mm>[l]^-{b}
}
\qquad\text{and}\qquad
I=\langle ab,ba\rangle .
\notag
\end{align}
Then $P(1)=\begin{smallmatrix}1\\2\end{smallmatrix}$ and $P(2)=\begin{smallmatrix}2\\1\end{smallmatrix}$. 
Let $\unlhd$ be the partial order on $\{1,2\}$ in which $1$ and $2$ are incomparable.
Then $\Delta^{\unlhd}(i)=S(i)$ for $i=1,2$, and hence $A\in\F(\Delta^{\unlhd})=\mod A$. 
However, $K^{\unlhd}(1)=\rad P(1)\cong S(2)$ does not lie in $\F(\emptyset)=\{0\}$. Thus $(A,\unlhd)$ is not standardly stratified. 
\end{example}

\subsection{Characteristic tilting modules}\label{subsec:tilting}
We recall basic results about tilting modules and introduce characteristic tilting modules.
For an $A$-module $M$, let $\cores(\add M)$ denote the subcategory of $\mod A$ consisting of modules $X$ admitting a finite $\add M$-coresolution, that is, an exact sequence $0\to X\to M^0\to\dots\to M^r\to0$ with $M^l\in\add M$.

\begin{definition}\label{def:tilting}
An $A$-module $T$ is called a \emph{tilting module} if the following conditions are satisfied. 
\begin{enumerate}[label=\textup{(T\arabic*)}]
\item $\pd T<\infty$.
\item $\Ext^m_A(T,T)=0$ for all $m\ge1$.
\item $A_A\in\cores(\add T)$, that is, there is an exact sequence $0\to A\to T^0\to T^1\to\cdots\to T^m\to0$ with $T^l\in\add T$ for all $l$.
\end{enumerate}
Let $\tilt A$ be the set of isomorphism classes of basic tilting $A$-modules.
\end{definition}

We collect some basic results about tilting modules that will be used throughout the paper.

\begin{proposition}\label{prop:tilting}
Let $T$ be a tilting $A$-module. Then the following statements hold.
\begin{enumerate}[label=\textup{(\arabic*)}]
\item \textup{(\cite[Theorem~5.5]{AR91})} $\cores(\add T)={}^{\perp}(T^{\perp})$ holds. 
In particular, $\cores(\add T)$ is closed under direct summands.
\item \textup{(\cite{HU})} For tilting $A$-modules $T$ and $U$, we write $U\ge T$ if $\Ext^{m}_A(U,T)=0$ for all $m\ge1$. 
 Then $\ge$ defines a partial order on $\tilt A$, and $U\ge T$ if and only if $T^{\perp}\subseteq U^{\perp}$.
\item In the Grothendieck group $K_0(\mod A)$, the classes of the indecomposable direct summands of $T$ generate the same subgroup as $[P(1)],\dots,[P(n)]$.
\end{enumerate}
\end{proposition}

Statement~(3) follows from \textup{(T1)}, \textup{(T3)} and (1) by taking alternating sums in finite projective resolutions of the summands of $T$ and in a finite $\add T$-coresolution of $P(i)$, which exists since $P(i)$ is a direct summand of $A_A$.

We recall the characteristic tilting module of a standardly stratified algebra. 
The corresponding results for total orders are due to \cite{AHLU,Xi}.
The general case follows from the total order case by Lemma~\ref{lem:refine} and \cite[Theorem~2.2.7]{CPS2}.
If $(A,\unlhd)$ is quasi-hereditary, the module $T_\unlhd$ below is Ringel's characteristic tilting module \cite{R} (see Remark~\ref{rem:costandard}).

\begin{theorem}\label{thm:chtilt}
Let $(A,\unlhd)$ be standardly stratified.
Then the following statements hold.
\begin{enumerate}[label=\textup{(\arabic*)}]
\item There is a basic tilting module $T_\unlhd$ with $\add T_\unlhd=\F(\Delta)\cap\Y(\Delta)$. 
\item There is a decomposition $T_\unlhd=\bigoplus_{i\in\Lambda}T_\unlhd(i)$ into pairwise non-isomorphic indecomposable modules, and, for every $i\in\Lambda$, an exact sequence
\[
0\to\Delta(i)\xrightarrow{\beta_i}T_\unlhd(i)\to X(i)\to0
\]
such that $X(i)\in\F(\Delta(\lhd i))$ and $\beta_i$ is a minimal left $\Y(\Delta)$-approximation of $\Delta(i)$.
In particular, $\supp T_\unlhd(i)\subseteq\{j\mid j\unlhd i\}$.
\end{enumerate}
\end{theorem}

We call $T_\unlhd$ the \emph{characteristic tilting module} associated with $(A,\unlhd)$. 
By uniqueness of minimal left $\Y(\Delta)$-approximations, its labelled isomorphism class depends only on the labelled family $\Delta$.

The characteristic tilting module determines both $\F(\Delta)$ and $\Y(\Delta)$.
The following proposition is due to \cite{AHLU,Xi} for total orders, and the general case reduces to this case by Lemma~\ref{lem:refine}. 

\begin{proposition}\label{prop:orthogonal}
If $(A,\unlhd)$ is standardly stratified and $T=T_\unlhd$, then
\[
\F(\Delta)
=
\cores(\add T)
=
{}^{\perp}(T^\perp),
\qquad
\Y(\Delta)
=
T^\perp.
\]
In particular, $\Ext_A^m(X,Y)=0$ for all $X\in\F(\Delta)$, $Y\in\Y(\Delta)$, and $m\ge1$.
\end{proposition}

\begin{remark}\label{rem:costandard}
By \cite{AHLU, Fr2}, we have $\Y(\Delta)=\F(\bar\nabla)$, and hence $\add T_\unlhd=\F(\Delta)\cap\F(\bar\nabla)$.
If $(A,\unlhd)$ is quasi-hereditary, then $\bar\nabla(i)$ coincides with the
costandard module $\nabla(i)$, the maximal submodule of $I(i)$ whose support is contained in $\{j\mid j\unlhd i\}$. 
In this case, $\add T_\unlhd=\F(\Delta)\cap\F(\nabla)$ is the classical result of Ringel \cite[Theorem~5]{R}.
\end{remark}

\section{A reject criterion for characteristic tilting modules}\label{sec:rejectchar}\label{sec:reject}

In this section, we characterize characteristic tilting modules in terms of rejects.
Throughout this section, $\Lambda$ is endowed with the natural order.

\subsection{Costratifying systems from self-orthogonal modules}
In this subsection, we give sufficient conditions for a self-orthogonal module and its reject submodules to form a costratifying system.
We begin with an elementary property of rejects, which will be used repeatedly in this paper. 

\begin{lemma}\label{lem:rej-approx}
Let $M,N\in\mod A$, and let $g\colon M\to W$ be a homomorphism with $W\in\add N$.
Then the following statements hold. 
\begin{enumerate}[label=\textup{(\arabic*)}]
\item $\rej_N(M)\subseteq\Ker g$.
\item If $g$ is a left $\add N$-approximation, then $\Ker g=\rej_N(M)$.
\end{enumerate}
\end{lemma}
\begin{proof}
\textup{(1)} Let $\iota\colon W\to N^{\oplus r}$ be a split monomorphism for some $r$.
Every component of $\iota g\colon M \to N^{\oplus r}$ vanishes on $\rej_N(M)$ by the definition of the reject.
Since $\iota$ is a monomorphism, $g$ vanishes on $\rej_N(M)$. 
Thus $\rej_N(M)\subseteq\Ker g.$

\textup{(2)}
Since every homomorphism $\alpha\colon M \to W'$ with $W'\in \add N$ factors through $g$, we have $\Ker g\subseteq \Ker \alpha$. 
Hence $\Ker g\subseteq\rej_N(M)$, and the reverse inclusion follows from (1).  
\end{proof}

By a \emph{labelled decomposition} of an $A$-module $T$, we mean a
decomposition
\[
T=\bigoplus_{i\in\Lambda}T(i)
\]
into indecomposable modules indexed by $\Lambda$.
The direct summands need not be pairwise non-isomorphic.

We consider the following reject submodules.

\begin{definition}\label{def:theta}
Let $T=\bigoplus_{i\in\Lambda}T(i)$ be a labelled decomposition.
Put
\[
T_{<i}=\bigoplus_{j<i}T(j),\qquad
\Theta_T(i)=\rej_{T_{<i}}(T(i)),\qquad
C_T(i)=T(i)/\Theta_T(i).
\]
Thus $T_{<1}=0$ and $\Theta_T(1)=T(1)$.
We write
$\Theta_T=(\Theta_T(1),\dots,\Theta_T(n))$
and consider the following conditions.
\begin{enumerate}[label=\textup{(C\arabic*)}]
\item $\Top\Theta_T(i)\cong S(i)$ for every $i\in\Lambda$.
\item $C_T(i)\in\F(\Theta_T(<i))$ for every $i\in\Lambda$.
\end{enumerate}
The labelled decomposition is called \emph{reject-triangular} if \textup{(C1)} and \textup{(C2)} hold.
\end{definition}

\begin{remark}\label{rem:C1}
\begin{enumerate}[label=\textup{(\arabic*)}]
\item Under \textup{(C1)}, each $\Theta_T(i)$ is nonzero with simple top, hence it is indecomposable.
\item If $\Theta_T(j)\ne0$ for every $j$, in particular under \textup{(C1)}, then $T$
is basic. Indeed, if $T(i)\cong T(j)$ for some $i<j$, then $T(j)$ embeds into
$T_{<j}$, and hence $\Theta_T(j)=0$ by Lemma~\ref{lem:rej-approx}(1).
\item Under \textup{(C2)}, $T(i)\in\F(\Theta_T(\le i))$ for every $i$. In particular, $T\in\F(\Theta_T)$.
\end{enumerate}
\end{remark}

The following proposition constructs a costratifying system from the reject submodules $\Theta_T(i)$.

\begin{proposition}\label{prop:cost}
Let $T$ be an $A$-module satisfying \textup{(T2)}, and fix a labelled decomposition $T=\bigoplus_{i=1}^n T(i)$.
Assume that the labelled decomposition satisfies \textup{(C2)} and that
$\Theta_T(i)$ is indecomposable for every $i$. Then the following statements hold.
\begin{enumerate}[label=\textup{(\arabic*)}]
\item $\Ext^m_A(X,T)=0$ for every $X\in\F(\Theta_T)$ and every $m\ge1$.
\item $\Hom_A(\Theta_T(j),T(i))=0$ whenever $i<j$. In particular, $\Theta_T$ satisfies \textup{(S2)}.
\item $(\Theta_T,(T(1),\dots,T(n)))$ is a costratifying system. Consequently, $\Theta_T$ is a standardizable set, $T$ is an injective cogenerator of $\F(\Theta_T)$, and $\add T=\F(\Theta_T)\cap\Y(\Theta_T)$.
\end{enumerate}
\end{proposition}

\begin{proof}
Write $\Theta=\Theta_T$ and $C(i)=C_T(i)$. Then we have an exact sequence
\begin{equation}\label{eq:TC}
0\to\Theta(i)\to T(i)\to C(i)\to0
\end{equation}
such that $C(i)\in\F(\Theta(<i))$.

\textup{(1)} By induction on filtration length, it suffices to show $\Ext^m_A(\Theta(i),T)=0$ for all $m\ge1$ and $i\in\Lambda$. We use induction on $i$. 
For $i=1$, this follows from (T2) since $\Theta(1)=T(1)$. Let $i\ge2$ and assume the assertion holds for all $j<i$. Then $\Ext^m_A(C(i),T)=0$ for all $m\ge1$. 
Applying $\Hom_A(-,T)$ to \eqref{eq:TC} yields an exact sequence 
\[\Ext^m_A(T(i),T)\to\Ext^m_A(\Theta(i),T)\to\Ext^{m+1}_A(C(i),T)\] 
whose outer terms vanish for $m\ge1$.
Hence $\Ext^m_A(\Theta(i),T)=0$ for $m\ge1$. 

\textup{(2)} Let $i<j$ and $g\colon\Theta(j)\to T(i)$. By (1), $\Ext^1_A(C(j),T(i))=0$. 
Thus applying $\Hom_A(-,T(i))$ to \eqref{eq:TC} for $j$ shows that $g$ extends to a homomorphism $\tilde g\colon T(j)\to T(i)$. 
Since $T(i)\in\add T_{<j}$, it follows from Lemma~\ref{lem:rej-approx}(1) that $\tilde g$
vanishes on $\Theta(j)=\rej_{T_{<j}}T(j)$. Hence $g=0$.
Since $\Theta(i)\subseteq T(i)$, we obtain $\Hom_A(\Theta(j),\Theta(i))=0$ whenever $i<j$.

\textup{(3)} By the assumption and (2), $\Theta$ satisfies (S1) and (S2). The modules $T(i)$ are indecomposable and lie in $\F(\Theta)^{\perp_1}=\Y(\Theta)$ by (1), and \eqref{eq:TC} is an exact sequence with $C(i)\in\F(\Theta(<i))$. Thus $(\Theta,(T(1),\dots,T(n)))$ is a costratifying system, and the remaining assertions follow from Proposition~\ref{prop:AT}(2).
\end{proof}

\subsection{Characterization by reject-triangularity}

The following theorem is the main result of this section.
Among the tilting conditions, it assumes only \textup{(T2)} and \textup{(T3)}.

\begin{theorem}\label{thm:reject}
Let $T=\bigoplus_{i=1}^nT(i)$ be a labelled decomposition of an $A$-module $T$. Then the following conditions are equivalent.
\begin{enumerate}[label=\textup{(\roman*)}]
\item $(A,\le)$ is standardly stratified for the natural order $\le$, and $T(i)\cong T_\le(i)$ for every $i\in\Lambda$.
\item $T$ satisfies \textup{(T2)} and \textup{(T3)}, and the labelled decomposition is reject-triangular.
\end{enumerate}
If these conditions hold, then $\Theta_T(i)\cong\Delta(i)$ for every $i\in\Lambda$, and $T$ is a tilting module.
\end{theorem}

The following lemma identifies the reject submodules of a characteristic tilting module. 

\begin{lemma}\label{lem:chtilt-rej}
Let $(A,\le)$ be standardly stratified for the natural order $\le$, and let $T=T_\le$ with the exact sequence $0\to\Delta(i)\xrightarrow{\beta_i}T(i)\to X(i)\to0$ of Theorem~\ref{thm:chtilt}(2). Then the following statements hold for every $i$.
\begin{enumerate}[label=\textup{(\arabic*)}]
\item $\Theta_T(i)=\beta_i(\Delta(i))$ and $C_T(i)\cong X(i)\in\F(\Delta(<i))$.
\item Let $g_i\colon T(i)\to W_i$ be a minimal left $\add T_{<i}$-approximation. Then $\Ker g_i=\Theta_T(i)\cong\Delta(i)$, and both $\Img g_i$ and $\Cok g_i$ belong to $\F(\Delta(<i))$.
\end{enumerate}
\end{lemma}

\begin{proof}
Let $q_i\colon T(i)\to X(i)$ be the cokernel of $\beta_i$.
By Lemma~\ref{lem:Delta} and Theorem~\ref{thm:chtilt}, $(\Delta(<i),(T(1),\dots,T(i-1)))$ is a costratifying system, and hence $T_{<i}$ is an injective cogenerator of $\F(\Delta(<i))$ by Proposition~\ref{prop:AT}(2).
Since $X(i)\in\F(\Delta(<i))$, there is an exact sequence $0\to X(i)\xrightarrow{u}Y'\to M'\to0$ with $Y'\in\add T_{<i}$ and $M'\in\F(\Delta(<i))$.
Put $h=uq_i$.
We claim that $h$ is a left $\add T_{<i}$-approximation.
Let $\alpha\colon T(i)\to Y''$ with $Y''\in\add T_{<i}$.
Since $T(j)\in\F(\Delta(\le j))$ for every $j$, we have $T_{<i}\in\F(\Delta(<i))$, and hence $\Hom_A(\Delta(i),T_{<i})=0$ by Lemma~\ref{lem:Delta}(3).
Since $\alpha\beta_i=0$, there is $\varphi\colon X(i)\to Y''$ such that $\alpha=\varphi q_i$.
Since $\Ext^1_A(M',Y'')=0$ by Theorem~\ref{thm:chtilt}(1), there is $\widetilde\varphi\colon Y'\to Y''$ with $\varphi=\widetilde\varphi u$, and hence $\alpha=\widetilde\varphi h$.

Since $g_i$ is a minimal left $\add T_{<i}$-approximation, there is a decomposition $Y'=Y'_1\oplus Y'_2$ such that $h=\left(\begin{smallmatrix}h_1\\0\end{smallmatrix}\right)$ and $h_1$ is isomorphic to $g_i$.
Hence $\Ker g_i=\Ker h=\beta_i(\Delta(i))$ and $\Img g_i\cong\Img h=u(X(i))\cong X(i)\in\F(\Delta(<i))$.
Since $\Cok g_i$ is a direct summand of $\Cok h\cong M'$, we have $\Cok g_i\in\F(\Delta(<i))$ by Proposition~\ref{prop:AT}(1).
Finally, $\Theta_T(i)=\rej_{T_{<i}}T(i)=\Ker g_i$ by Lemma~\ref{lem:rej-approx}(2). Thus $\Theta_T(i)=\beta_i(\Delta(i))\cong\Delta(i)$ and $C_T(i)=T(i)/\beta_i(\Delta(i))\cong X(i)$.
\end{proof}

The proof of the implication (ii)$\Rightarrow$(i) in Theorem~\ref{thm:reject} uses
the fact that $\F(\Theta_T)$ is closed under kernels of epimorphisms. We establish
this property for standardizable sets $\Theta$ such that $\Top\Theta(i)$ and
$\Top\Theta(j)$ have no nonzero common direct summand for $i\ne j$. We first record
an elementary lemma.

\begin{lemma}\label{lem:split}
Let $U$ be an indecomposable $A$-module and let $f\colon U^{\oplus b}\to U^{\oplus c}$ be an epimorphism. Then $f$ is a split epimorphism and $\Ker f\cong U^{\oplus (b-c)}$.
\end{lemma}

\begin{proof}
Let $E=\End_A(U)$ and $J=\rad E$. Since $U$ is indecomposable, $E$ is local and $J$
is nilpotent. Put $J^mU=\sum_{\varphi\in J^m}\varphi(U)$ for $m\ge1$.
If $JU= U$, then $U=JU=J^2U=\dots=J^mU=0$ for $m\gg0$, a contradiction.
 Thus $JU\ne U$. 
We proceed by induction on $c$. 
If $c=0$, then the assertion is clear.  
Let $c\ge1$, and write
$f=(f_{lm})$ with $f_{lm}\in E$. If all $f_{lm}$ belong to $J$, then
$\Img f\subseteq(JU)^{\oplus c}\ne U^{\oplus c}$, which contradicts the assumption that $f$ is an epimorphism.
 Hence some $f_{lm}$ is an automorphism of $U$, and we may assume that $f_{11}$
is an automorphism. Elementary row and column operations yield automorphisms
$\sigma$ of $U^{\oplus c}$ and $\tau$ of $U^{\oplus b}$ such that
$
\sigma f\tau=\begin{pmatrix}1_U&0\\0&f'\end{pmatrix}
$
for some homomorphism $f'\colon U^{\oplus(b-1)}\to U^{\oplus(c-1)}$. Since $f'$ is an
epimorphism, $f'$ is a split epimorphism with $\Ker f'\cong U^{\oplus(b-c)}$ by the induction hypothesis. 
Since $\sigma$ and $\tau$ are automorphisms, it follows
that $f$ is a split epimorphism and $\Ker f\cong\Ker f'\cong U^{\oplus(b-c)}$.
\end{proof}

\begin{proposition}\label{thm:kernel}
Let $\Theta=(\Theta(1),\dots,\Theta(t))$ be a standardizable set in $\mod A$ such that, for $i\neq j$, the modules $\Top\Theta(i)$ and $\Top\Theta(j)$ have no nonzero common direct summand. Then $\F(\Theta)$ is closed under kernels of epimorphisms.
\end{proposition}

\begin{proof}
We use induction on $t$. If $t=0$, then $\F(\Theta)=\{0\}$ and there is nothing to show. Let $t\ge1$ and note that $\Theta(<t)=(\Theta(1),\dots,\Theta(t-1))$ satisfies the same assumptions as $\Theta$. Let $f\colon M\to N$ be an epimorphism with $M,N\in\F(\Theta)$. By Proposition~\ref{prop:AT}(1), there are exact sequences $0\to M_1\to M\to M_2\to0$ and $0\to N_1\to N\to N_2\to0$ with $M_1,N_1\in\add\Theta(t)$ and $M_2,N_2\in\F(\Theta(<t))$. By (S2), $\Hom_A(M_1,N_2)=0$, and hence we obtain a commutative diagram with exact rows
\[
\begin{array}{ccccccccc}
0&\longrightarrow&M_1&\longrightarrow&M&\longrightarrow&M_2&\longrightarrow&0\\
&&\big\downarrow{\scriptstyle f_1}&&\big\downarrow{\scriptstyle f}&&\big\downarrow{\scriptstyle f_2}&&\\
0&\longrightarrow&N_1&\longrightarrow&N&\longrightarrow&N_2&\longrightarrow&0
\end{array}
\]
in which $f_2$ is an epimorphism since $f$ is. By the induction hypothesis, $\Ker f_2\in\F(\Theta(<t))$. Since $\Cok f=0$, the snake lemma gives an exact sequence
\[
0\to\Ker f_1\to\Ker f\to\Ker f_2\to\Cok f_1\to0 .
\]
Put $V=\Cok f_1$. Since $V$ is a factor module of $N_1\in\add\Theta(t)$ and
$\Top=-\otimes_AA/\rad A$ is right exact, every simple direct summand of $\Top V$
occurs in $\Top\Theta(t)$.
On the other hand, $V$ is also a factor module of $\Ker f_2\in\F(\Theta(<t))$. 
Thus every simple direct summand of $\Top V$ also occurs in $\Top\Theta(j)$ for some $j<t$. 
By assumption, $\Top\Theta(t)$ and $\Top\Theta(j)$ have no nonzero common direct summand for $j<t$. 
Hence $\Top V=0$, and hence $V=0$. Thus $f_1$ is an epimorphism.

Since $f_1$ is an epimorphism between modules isomorphic to $\Theta(t)^{\oplus b}$ and $\Theta(t)^{\oplus c}$ for some $b,c\ge0$, it follows from Lemma~\ref{lem:split} that $\Ker f_1\in\add\Theta(t)$. Since $0\to\Ker f_1\to\Ker f\to\Ker f_2\to0$ is exact with outer terms in $\F(\Theta)$, we obtain $\Ker f\in\F(\Theta)$.
\end{proof}

The following example shows that the assumption in Proposition~\ref{thm:kernel} cannot be omitted.

\begin{example}\label{ex:C1}
Let $A=k(1\to2)$. 
Then $P(1)=\begin{smallmatrix}1\\2\end{smallmatrix}$, $P(2)=S(2)$, and $T=P(1)\oplus S(1)\cong \DD A$ is a tilting module.
Consider the labelled decomposition $T(1)=P(1)$, $T(2)=S(1)$. 
Since $\Hom_A(S(1),P(1))=0$, we have $\Theta_T(1)=P(1)$, $\Theta_T(2)=S(1)$ and $C_T(1)=C_T(2)=0$.
Thus \textup{(C2)} holds, whereas \textup{(C1)} fails.

By Proposition~\ref{prop:cost}(3), $\Theta_T$ is a standardizable set.
However, $\Top \Theta_T(i)\cong S(1)$ for $i=1,2$ and
$\F(\Theta_T)=\add(P(1)\oplus S(1))$ is not closed under kernels
of epimorphisms. Indeed, the canonical epimorphism
$P(1)\to S(1)$ has kernel $S(2)\notin\F(\Theta_T)$.
\end{example}

\begin{remark}
For any total order on $\Lambda$, the standard modules form
a standardizable set with pairwise non-isomorphic simple tops by Lemma~\ref{lem:Delta}.
 Thus Proposition~\ref{thm:kernel} recovers the fact that $\mathcal{F}(\Delta)$ is closed under kernels of epimorphisms
(\cite[Lemma~1.5]{DR}, \cite[Proposition~2.1]{Xi}).
\end{remark}

We now prove Theorem~\ref{thm:reject}.

\begin{proof}[Proof of Theorem~\ref{thm:reject}]
We show (i)$\Rightarrow$(ii). We may assume $T(i)=T_\le(i)$ for all $i$. 
By Theorem~\ref{thm:chtilt}, $T$ is a tilting module, and hence (T2) and (T3) hold. By Lemma~\ref{lem:chtilt-rej}(1), $\Theta_T(i)\cong\Delta(i)$, whose top is $S(i)$, and $C_T(i)\in\F(\Delta(<i))=\F(\Theta_T(<i))$.

We show (ii)$\Rightarrow$(i). By Remark~\ref{rem:C1}(1), every $\Theta_T(i)$ is indecomposable. 
Hence, by Proposition~\ref{prop:cost}(3) and (C1), $\Theta:=\Theta_T$ is a standardizable set with $\Top\Theta(i)\cong S(i)$ for all $i$.

\emph{Step 1: $\Theta(i)\cong\Delta(i)$ for all $i$, and $(A,\le)$ is standardly stratified.} 
By (C1), the modules $\Top\Theta(i)\cong S(i)$ are pairwise non-isomorphic simple modules, and hence $\F(\Theta)$ is closed under kernels of epimorphisms by Proposition~\ref{thm:kernel}. 
Since $\add T\subseteq\F(\Theta)$ by Remark~\ref{rem:C1}(3), it follows from \textup{(T3)} by descending induction that $A\in\F(\Theta)$.
Hence the assertion follows from Lemma~\ref{lem:Theta-Delta}.

\emph{Step 2: $T(i)\cong T_\le(i)$ for all $i$.} 
By Step~1, $\F(\Theta)=\F(\Delta)$ and $\F(\Theta(<i))=\F(\Delta(<i))$. 
Let $\gamma_i\colon\Delta(i)\to T(i)$ be the composite of an isomorphism $\Delta(i)\cong\Theta(i)$ with the inclusion $\Theta(i)\subseteq T(i)$. 
Then there is an exact sequence $0\to\Delta(i)\xrightarrow{\gamma_i}T(i)\to C_T(i)\to0$ with $C_T(i)\in\F(\Delta)$, and $T(i)\in\Y(\Delta)$ by Proposition~\ref{prop:cost}(1). 
For $Y\in\Y(\Delta)$, we have $\Ext^1_A(C_T(i),Y)=0$, and hence $\Hom_A(\gamma_i,Y)$ is surjective. 
Thus $\gamma_i$ is a left $\Y(\Delta)$-approximation. 
It is left minimal since $T(i)$ is indecomposable and $\gamma_i\neq0$.
 By Theorem~\ref{thm:chtilt}(2), $\beta_i\colon\Delta(i)\to T_\le(i)$ is also a minimal left $\Y(\Delta)$-approximation. 
 Hence $T(i)\cong T_\le(i)$ by the uniqueness of minimal approximations. In particular, $T\cong T_\le$ is a tilting module.
\end{proof}

We finish this section with a consequence of Proposition~\ref{prop:cost} for endomorphism algebras.

\begin{corollary}\label{cor:ringeldual}
Under the assumptions of Proposition~\ref{prop:cost}, label the indecomposable projective module $\Hom_A(T(i),T)$ over $\Gamma=\End_A(T)^{\op}$ by $i$. Then $\Gamma$ is standardly stratified with respect to the order $n<n-1<\dots<1$, and its standard module labelled by $i$ is $\Hom_A(\Theta_T(i),T)$.
\end{corollary}
\begin{proof}
Under the assumptions of Proposition~\ref{prop:cost}, the module $T$ is basic by Remark~\ref{rem:C1}(2).
Thus the assertion follows from Proposition~\ref{prop:cost}(3) and \cite[Proposition~1.3]{ES}.
\end{proof}

\begin{remark}\label{rem:ringeldual}
Corollary~\ref{cor:ringeldual} may be viewed as the standardly stratified analogue of
\cite[Proposition~3.17(b)]{ACKT}. Note that conditions \textup{(T1)} and \textup{(T3)} are not assumed in Corollary~\ref{cor:ringeldual}.
If $(A,\le)$ is standardly stratified and $T=T_\le$, then the assumptions of
Proposition~\ref{prop:cost} are satisfied by Lemma~\ref{lem:chtilt-rej}(1), and
Corollary~\ref{cor:ringeldual} recovers Ringel duality for standardly stratified
algebras \cite{AHLU,Xi}.
\end{remark}

\section{A support criterion for characteristic tilting modules}\label{sec:main}

In this section, we give a characterization of characteristic tilting modules in terms of supports. 

\subsection{Triangular tilting modules}
In this subsection, we introduce triangular tilting modules, state the main theorem of this section, and give its consequences.

\begin{definition}\label{def:intro}
Let $T=\bigoplus_{i\in\Lambda}T(i)$ be a labelled decomposition.
It is called \emph{triangular} if there exists a partial order $\unlhd$ on $\Lambda$ such that $\supp T(i)\subseteq\{j\in\Lambda\mid j\unlhd i\}$ for every $i\in\Lambda$. 
In this case, we also say that the labelled decomposition is \emph{triangular with respect to} $\unlhd$.
A triangular labelled decomposition is called \emph{unitriangular} if $[T(i):S(i)]=1$ for every $i$. 
A basic tilting module is called \emph{triangular}, respectively
\emph{unitriangular}, if it admits a triangular, respectively
unitriangular, labelled decomposition.
\end{definition}

The notion of a unitriangular tilting module was introduced in the submitted version of our joint work \cite{ACKT}.

By Theorem~\ref{thm:chtilt}(2), if $(A,\unlhd)$ is standardly stratified, then its characteristic tilting module $T_{\unlhd}$ satisfies $\supp T_\unlhd(i)\subseteq\{j\mid j\unlhd i\}$ for every $i$.
Hence characteristic tilting modules are triangular. 

A labelled decomposition is triangular if and only if it is triangular with respect to some total order.
Indeed, if $\unlhd$ satisfies the support condition above and $\le$ is a linear extension of $\unlhd$, then $\{j\in\Lambda\mid j\unlhd i\}\subseteq\{j\in\Lambda\mid j\le i\}$ for every $i\in\Lambda$, and hence the same support condition holds for $\le$.
The converse is immediate.

Consequently, a labelled decomposition is triangular precisely when, after arranging the rows and columns according to a suitable total order, the matrix $\bigl([T(i):S(j)]\bigr)_{i,j}$ of composition multiplicities is lower triangular.
It is unitriangular precisely when this matrix can be made lower triangular with all diagonal entries equal to one.

The following theorem is one of the main results of this paper. 
It characterizes characteristic tilting modules in terms of the supports of their indecomposable summands.

\begin{theorem}\label{thm:main}
Let $A$ be a basic finite dimensional algebra over a field $k$.
Let $\unlhd$ be a partial order on $\Lambda$, and let $T=\bigoplus_{i\in\Lambda}T(i)$ be a labelled decomposition of a tilting $A$-module. Then the following conditions are equivalent.
\begin{enumerate}[label=\textup{(\roman*)}]
\item $(A,\unlhd)$ is standardly stratified, and $T(i)\cong T_\unlhd(i)$ for every $i\in\Lambda$.
\item The labelled decomposition is triangular with respect to $\unlhd$, that is, $\supp T(i)\subseteq\{j\mid j\unlhd i\}$ for every $i\in\Lambda$.
\end{enumerate}
In this case, the following statements hold for every $i\in\Lambda$.
\begin{enumerate}[label=\textup{(\alph*)}]
\item $\Delta(i)\cong T(i)e_iA= \tr_{P(i)}(T(i))$.
\item $T(i)e_i\cong D_i$ as right $D_i$-modules. In particular, $[T(i):S(i)]=c_i$.
\end{enumerate}
\end{theorem}

The implication from (i) to (ii), together with assertions (a) and (b), follows from Theorem~\ref{thm:chtilt} and Lemma~\ref{lem:Delta}.

\begin{remark}\label{rem:reject-support}
Let $T=\bigoplus_{i=1}^nT(i)$ be a labelled decomposition of a tilting $A$-module.
By Theorems~\ref{thm:reject} and~\ref{thm:main}, the labelled decomposition is
reject-triangular if and only if it is triangular with respect to the natural order.
\end{remark}

We prove Theorem~\ref{thm:main} in Subsection~\ref{subsec:proof-main}.
We first discuss its consequences.

For total orders, Theorem~\ref{thm:main} admits the following formulation in terms of idempotent subalgebras.

\begin{corollary}\label{cor:layers}
Let $\le$ be the natural order, and put $\varepsilon_i=\sum_{j\ge i}e_j$, $\varepsilon_{n+1}=0$, $A_i=\varepsilon_iA\varepsilon_i$, and $L_i=A_i/A_i\varepsilon_{i+1}A_i$.
Let $T=\bigoplus_{i=1}^nT(i)$ be a labelled decomposition of a tilting $A$-module.
Then the following conditions are equivalent.
\begin{enumerate}[label=\textup{(\roman*)}]
\item The labelled decomposition is triangular with respect to $\le$.
\item $T(i)\varepsilon_i\cong L_i$ as right $A_i$-modules for every $i$.
\item $T(i)e_j=0$ for $j>i$, and $T(i)e_i\cong D_i$ as right $e_iAe_i$-modules for every $i$.
\end{enumerate}
\end{corollary}

\begin{proof}
Put $J_i=A_i\varepsilon_{i+1}A_i$.
Since $\varepsilon_i=e_i+\varepsilon_{i+1}$ and
$\varepsilon_{i+1}\in J_i$, the natural homomorphism
\[
\rho\colon e_iAe_i\hookrightarrow A_i \twoheadrightarrow L_i=A_i/J_i
\]
is surjective. Moreover, $\Ker \rho=e_iJ_ie_i=e_iA\varepsilon_{i+1}Ae_i=K(i)e_i.$
Hence $L_i\cong D_i$ as algebras.
In particular, $L_ie_i=L_i$ and $L_ie_j=0$
for $j>i$.

Assume that \textup{(i)} holds.
Then $T(i)\varepsilon_{i+1}=0$, and hence $T(i)\varepsilon_i=T(i)e_i$.
By Theorem~\ref{thm:main}(b), $T(i)e_i\cong D_i$ as right $D_i$-modules.
Thus $T(i)\varepsilon_i\cong L_i$ as right $A_i$-modules.
Hence \textup{(ii)} holds.

Assume that \textup{(ii)} holds.
For $j>i$, we have $T(i)e_j=(T(i)\varepsilon_i)e_j\cong L_ie_j=0.$
Moreover,
\[
T(i)e_i
=(T(i)\varepsilon_i)e_i
\cong L_ie_i
=L_i
\cong D_i
\]
as right $e_iAe_i$-modules.
Hence \textup{(iii)} holds.

Finally, \textup{(iii)} implies \textup{(i)}
by the definition of triangularity.
\end{proof}

If all $L_i$ are division rings, then $L_i\cong S(i)\varepsilon_i$ as right $A_i$-modules.
In this case, condition \textup{(ii)} coincides with condition \textup{(I1)} of \cite[Definition~3.4]{ACKT}.

Theorem~\ref{thm:main} yields the following recursive criterion for triangular tilting modules. 

\begin{corollary}\label{cor:algorithm}
Let $T$ be a basic tilting $A$-module.
Consider sequences $(x_r,T_r)_{r=1}^{\ell}$, where $x_r\in\Lambda$ and the $T_r$ are pairwise non-isomorphic indecomposable direct summands of $T$, satisfying
\[
\supp T_r\setminus\{x_1,\dots,x_{r-1}\}=\{x_r\}
\qquad(1\le r\le\ell).
\]
Then $T$ is triangular if and only if such a sequence of length $n$ exists.
Moreover, if $T$ is triangular, every such sequence of length $\ell<n$ can be extended to one of length $n$.

For any sequence of length $n$, the labelled decomposition $T(x_r)=T_r$ is triangular with respect to the total order $x_1<\dots<x_n$, and $T$ is the characteristic tilting module associated with this order.
\end{corollary}

\begin{proof}
Let $(x_r,T_r)_{r=1}^n$ be any sequence as in the statement.
Then the vertices $x_1,\dots,x_n$ are pairwise distinct.
Since $T$ is a basic tilting module, it has exactly $n$ indecomposable direct summands, and hence $T_1,\dots,T_n$ are all of them.
Moreover, $\supp T_r\subseteq\{x_1,\dots,x_r\}$ for every $r$.
Hence the labelled decomposition $T(x_r)=T_r$ is triangular with respect to the total order $x_1<\dots<x_n$.
By Theorem~\ref{thm:main}, $T$ is the characteristic tilting module associated with this order.

Conversely, assume that $T$ is triangular, and fix a triangular labelled decomposition $T=\bigoplus_{y\in\Lambda}T(y)$ with respect to a total order $\le$.
By Theorem~\ref{thm:main}(b), $y\in\supp T(y)$ for every $y$.
It suffices to show that every sequence $(x_r,T_r)_{r=1}^{\ell}$ as in the statement with $0\le\ell<n$ can be extended by one term.
Starting with the empty sequence and repeating this extension also yields a sequence of length $n$.
Put $X_0=\emptyset$ and $X_r=\{x_1,\dots,x_r\}$ for $1\le r\le\ell$.
We show by induction on $r$ that $T_r\cong T(x_r)$ for $1\le r\le\ell$.
Write $T_r\cong T(y)$.
If $r=1$, then $y\notin X_0$.
If $r\ge2$, the induction hypothesis and $T_r\not\cong T_s$ for $s<r$ give $y\notin X_{r-1}$.
As $y\in\supp T(y)$ and $\supp T_r\setminus X_{r-1}=\{x_r\}$, we obtain $y=x_r$.
Let $y$ be the smallest element of $\Lambda\setminus X_\ell$ with respect to $\le$.
Then $y\in\supp T(y)\subseteq\{z\mid z\le y\}\subseteq X_\ell\cup\{y\}$. 
Hence $\supp T(y)\setminus X_\ell=\{y\}$, and $T(y)\not\cong T_r$ for $r\le\ell$ since $y\notin X_\ell$. 
Thus $x_{\ell+1}=y$ and $T_{\ell+1}=T(y)$ extend the sequence.
\end{proof}

The following corollary extends a characterization of characteristic tilting modules in \cite{ACKT} to arbitrary fields.

\begin{corollary}\label{cor:QH}
Let $\unlhd$ be a partial order on $\Lambda$, and let $T=\bigoplus_{i\in\Lambda}T(i)$ be a labelled decomposition of a tilting $A$-module. For $i\in\Lambda$, put $\epsilon_i=\sum_{j\not\lhd i}e_j$. Then the following conditions are equivalent.
\begin{enumerate}[label=\textup{(\roman*)}]
\item $(A,\unlhd)$ is quasi-hereditary, and $T(i)\cong T_\unlhd(i)$ for every $i$.
\item The labelled decomposition is unitriangular with respect to $\unlhd$.
\item $T(i)\epsilon_i\cong S(i)\epsilon_i$ as $\epsilon_iA\epsilon_i$-modules for every $i$. 
\end{enumerate}
\end{corollary}

\begin{proof}
By Theorem~\ref{thm:main}, the labelled decomposition is triangular with respect to $\unlhd$ if and only if $(A,\unlhd)$ is standardly stratified and $T(i)\cong T_\unlhd(i)$ for every $i$.
In this case, $[T(i):S(i)]=1$ if and only if $D_i$ is a division ring by Lemma~\ref{lem:Delta}(1) and Theorem~\ref{thm:main}(b).
Thus the equivalence of (i) and (ii) holds. 

We show the equivalence of \textup{(ii)} and \textup{(iii)}.
Put $B_i=\epsilon_i A\epsilon_i$.
The exact functor $(-)\epsilon_i$ sends $S(j)$
to zero if $j\lhd i$, and otherwise to the
simple $B_i$-module $S(j)\epsilon_i$.
Hence, for every $A$-module $M$ and $j\not\lhd i$, we have $[M\epsilon_i:S(j)\epsilon_i]_{B_i}=[M:S(j)]_A.$
Suppose that \textup{(ii)} holds.
Since $\supp T(i)\subseteq\{j\mid j\unlhd i\}$ and $[T(i):S(i)]=1$, the module $T(i)\epsilon_i$ is isomorphic to $S(i)\epsilon_i$.
Thus \textup{(iii)} holds.
Conversely, suppose that \textup{(iii)} holds.
Since $[T(i):S(j)]_A=[T(i)\epsilon_i:S(j)\epsilon_i]_{B_i}$, we have $[T(i):S(i)]_A=1$ and $[T(i):S(j)]_A=0$ for
every $j\not\unlhd i$.
Hence $\supp T(i)\subseteq\{j\mid j\unlhd i\}$. 
Thus \textup{(ii)} holds.
\end{proof}

The following corollary shows that triangular tilting modules are unitriangular when $A$ has finite global dimension.
In particular, it recovers the known result that a standardly stratified algebra of finite global dimension is quasi-hereditary \cite[Theorem 2.4]{AHLU}, \cite[Proposition~2.8]{Xi}.

\begin{corollary}\label{cor:gldim}
Assume that $\gldim A<\infty$, and let $\unlhd$ be a partial order on $\Lambda$. If a labelled decomposition $T=\bigoplus_iT(i)$ of a tilting $A$-module is triangular with respect to $\unlhd$, then $(A,\unlhd)$ is quasi-hereditary and $T(i)\cong T_\unlhd(i)$ for every $i$. In particular, if $(A,\unlhd)$ is standardly stratified, then it is quasi-hereditary.
\end{corollary}

\begin{proof}
The classes $[S(1)],\dots,[S(n)]$ form a $\mathbb Z$-basis of $K_0(\mod A)$.
Since $\gldim A<\infty$, the classes $[P(1)],\dots,[P(n)]$ generate $K_0(\mod A)$.
By Proposition~\ref{prop:tilting}(3), the classes $[T(1)],\dots,[T(n)]$ also generate $K_0(\mod A)$.
Since $K_0(\mod A)\cong\mathbb Z^n$, these $n$ classes form a $\mathbb Z$-basis.
Choose a linear extension $\le$ of $\unlhd$ and arrange the indices according to $\le$.
The matrix $C_T=\bigl([T(i):S(j)]\bigr)_{i,j}$ of composition multiplicities expresses the classes $[T(i)]$ in the basis $([S(j)])$.
Since both families form $\mathbb Z$-bases of $K_0(\mod A)$, we have $C_T\in GL_n(\mathbb Z)$.
Since the labelled decomposition is triangular, $C_T$ is lower triangular.
Hence
\[
1=|\det C_T|
=\prod_{i\in\Lambda}[T(i):S(i)].
\]
It follows that $[T(i):S(i)]=1$ for every $i$.
Hence $(A,\unlhd)$ is quasi-hereditary and $T(i)\cong T_\unlhd(i)$ for every $i$ by Corollary~\ref{cor:QH}.
\end{proof}

The following example gives a triangular tilting module that is not unitriangular.

\begin{example}\label{ex:ss}
Let $A=k Q/I$, where
\[
Q:\quad
\xymatrix{
1\ar@(ul,dl)[]_{\alpha}\ar[r]^{a}&2
},
\qquad
I=\langle\alpha^2,\alpha a\rangle.
\]
Put $M=P(1)/aA$ and consider the labelled decomposition
\[
T=M\oplus P(1),
\qquad
T(1)=M,\quad T(2)=P(1).
\]
Then the module $T$ is tilting, $\supp T(1)=\{1\}$ and $\supp T(2)=\{1,2\}$. 
Hence $T$ is triangular with respect to $1<2$.
By Theorem~\ref{thm:main}, $(A,\le)$ is standardly
stratified and $T$ is its characteristic tilting module.
By Theorem~\ref{thm:main}(a), the standard modules are
\[
\Delta(1)\cong T(1)e_1A=M,
\qquad
\Delta(2)\cong T(2)e_2A=aA\cong S(2).
\]
Since
$\End_A(\Delta(1))\cong k [\alpha]/(\alpha^2)$,
the pair $(A,\le)$ is not quasi-hereditary.
Since $[M:S(1)]=2$ and $[M:S(2)]=0$, the module $T$ is triangular but not unitriangular.
\end{example}

The following proposition determines when $(A,\unlhd)$ is standardly stratified, provided that $(A,\le)$ is standardly stratified for a linear extension $\le$ of $\unlhd$.

\begin{proposition}\label{prop:descent}
Let $\unlhd$ be a partial order on $\Lambda$, and let $\le$ be a linear extension of $\unlhd$ such that $(A,\le)$ is standardly stratified. Put $T=T_\le$. Then $(A,\unlhd)$ is standardly stratified if and only if $\supp T(i)\subseteq\{j\mid j\unlhd i\}$ for every $i\in\Lambda$. 
In this case, $\Delta^\unlhd(i)=\Delta^\le(i)$, $K^\unlhd(i)=K^\le(i)$ and $D_i^\unlhd=D_i^\le$ for every $i\in\Lambda$, and $T_\unlhd=T_\le$.
\end{proposition}

\begin{proof}
Assume that $(A,\unlhd)$ is standardly stratified. 
Since $\le$ refines $\unlhd$, it follows from Lemma~\ref{lem:refine} that $\Delta^\unlhd(i)=\Delta^\le(i)$, $K^\unlhd(i)=K^\le(i)$ and $D^\unlhd_i=D^\le_i$. 
Since the standard modules coincide, so do the characteristic tilting modules. 
Hence $\supp T(i)\subseteq\{j\mid j\unlhd i\}$ for every $i\in\Lambda$  by Theorem~\ref{thm:chtilt}(2).

Conversely, assume the support condition. Fix $i\in\Lambda$.
Since $\le$ is a linear extension of $\unlhd$, every $j>i$ satisfies $j\not\unlhd i$, and hence $K^\le(i)\subseteq K^\unlhd(i)$.
By Theorem~\ref{thm:chtilt}(2), $\Delta^\le(i)$ is isomorphic to the  submodule $\beta_i(\Delta^\le(i))$ of $T(i)$, and hence $\supp\Delta^\le(i)\subseteq\{j\mid j\unlhd i\}$ by the support condition. By the maximality of $\Delta^\unlhd(i)$, we have $K^\unlhd(i)\subseteq K^\le(i)$.
Hence $K^\unlhd(i)=K^\le(i)$ and $\Delta^\unlhd(i)=\Delta^\le(i)$.

For simplicity, we write $\Delta(i)$ and $K(i)$. Choose a filtration $0=M_0\subseteq M_1\subseteq\dots\subseteq M_r=K(i)$ with $M_s/M_{s-1}\cong\Delta(j_s)$ and $j_s>i$. 
We show $i\lhd j_s$ for every $s$.
For each $s$, the module $P(i)/M_s$ lies in $\F(\Delta)$ and $T(j_s)\in\Y(\Delta)$ by Theorem~\ref{thm:chtilt}(1). Hence $\Ext_A^1(P(i)/M_s,T(j_s))=0.$ 
The composite $M_s\twoheadrightarrow M_s/M_{s-1}\cong\Delta(j_s)\xrightarrow{\beta_{j_s}}T(j_s)$ is nonzero.
Since this composite extends to a nonzero homomorphism $P(i)\to T(j_s)$, we have $i\in\supp T(j_s)$. 
By the assumption, $i\unlhd j_s$ and $i\ne j_s$, and hence $i\lhd j_s$. 
Thus $K(i)\in\F(\Delta(\rhd i))$, and $(A,\unlhd)$ is standardly stratified.
\end{proof}

\subsection{Proof of the support criterion}\label{subsec:proof-main}
In this subsection, we study tilting modules over idempotent subalgebras and prove Theorem~\ref{thm:main}.

Let $\Gamma\subseteq\Lambda$, $\varepsilon=\sum_{x\in\Gamma}e_x$ and $B=\varepsilon A\varepsilon$.
The exact functor $(-)\varepsilon=\Hom_A(\varepsilon A,-)\colon\mod A\to\mod B$ has the left adjoint $L=-\otimes_B\varepsilon A$.
Let $\PP_\Gamma$ be the subcategory of $\mod A$ consisting of modules admitting a projective resolution with all terms in $\add\varepsilon A$.
Let $\QQ_\Gamma$ be the subcategory of $\mod B$ consisting of modules $Y$ such that $\Tor_j^B(Y,\varepsilon A)=0$ for all $j\ge1$.

We recall the following comparison results from \cite{APT}.

\begin{lemma}\label{lem:APT} The following statements hold. 
\begin{enumerate}[label=\textup{(\arabic*)}]
\item \textup{(\cite[Proposition~2.4]{APT})}
We have
\[
\PP_\Gamma
=
\{X\in\mod A\mid
\Ext_A^j(X,S(y))=0
\text{ for all }j\ge0
\text{ and }y\notin\Gamma\}.
\]
In particular, $\PP_\Gamma$ is closed under extensions, direct summands
and kernels of epimorphisms.

\item \textup{(\cite[Lemma~3.1(1), Theorem~3.6(b), Proposition~3.7(b) and Proposition~3.10]{APT})}
The functor $(-)\varepsilon$ restricts to an equivalence $\PP_\Gamma\xrightarrow{\sim}\QQ_\Gamma$ with quasi-inverse $L$.
Moreover, for every $X\in\PP_\Gamma$,
$N\in\mod A$ and $j\ge0$, there is an isomorphism $\Ext_A^j(X,N)\cong\Ext_B^j(X\varepsilon,N\varepsilon).$
\end{enumerate}
\end{lemma}

We study modules supported at a single vertex.

\begin{proposition}\label{prop:bottom}
Let $x\in\Lambda$ and let $M$ be an $A$-module such that $\supp M=\{x\}$, $\pd_AM<\infty$ and $\Ext^m_A(M,M)=0$ for $1\le m\le\pd_AM$. Then $M\cong \bigl(A/A(1-e_x)A\bigr)^{\oplus r}$ for some $r\ge1$. In particular, every indecomposable direct summand $T'$ of a tilting $A$-module with $\supp T'=\{x\}$ is isomorphic to $A/A(1-e_x)A$.
\end{proposition}

\begin{proof}
Put $\varepsilon=1-e_x$ and $I=A\varepsilon A$. 
Let $\cdots\to Q_1\xrightarrow{d_1} Q_0\xrightarrow{d_0} M\to0$ be a minimal projective resolution of $M$, and put $p=\pd_A M$.
Since $\Top M$ is a nonzero direct sum of copies of $S(x)$, we have $Q_0\cong P(x)^{\oplus r}$ for some $r\ge1$.
Let $d$ be the largest integer such that $P(x)$ occurs as a direct summand of $Q_d$.
We show $d=0$.
Suppose $d\ge1$. 
By the minimality of the resolution, $\Ext_A^d(M,S(x))\cong\Hom_A(Q_d,S(x))\ne0.$
On the other hand, $Q_{d+1}$ has no direct summand isomorphic to $P(x)$.
Hence $\Hom_A(Q_{d+1},Y)=0$ for every $A$-module $Y$ with $\supp Y\subseteq\{x\}$.
Take an exact sequence $0\to K\to M\to S(x)\to0.$ 
Since $\supp K\subseteq\{x\}$, we have $\Hom_A(Q_{d+1},K)=0$, and hence $\Ext_A^{d+1}(M,K)=0$.
Applying $\Hom_{A}(M,-)$ to $0\to K\to M\to S(x)\to0$ yields an epimorphism
\[
\Ext_A^d(M,M)\longrightarrow
\Ext_A^d(M,S(x)).
\]
This contradicts the assumptions, since
$1\le d\le p$ and $\Ext_A^d(M,S(x))\ne0$.
Thus $d=0$.

Since $d=0$, we have $Q_1\in\add\varepsilon A$.
Put $\Omega=\Ker d_0$.
Since $Q_1=Q_1\varepsilon A$ and $\Omega=\Img d_1$, we have $\Omega=\Omega\varepsilon A\subseteq Q_0\varepsilon A=Q_0I.$
Conversely, $\supp M=\{x\}$ implies $MI=0$, and hence $Q_0I\subseteq\Omega$.
Since $\Omega=Q_0I$, we have 
\[M\cong Q_0/Q_0I\cong (P(x)/P(x)I)^{\oplus r}\cong (A/I)^{\oplus r}.\]

Finally, let $T'$ be an indecomposable direct summand of a tilting $A$-module with $\supp T'=\{x\}$.
Then $\pd_A T'<\infty$ and $\Ext_A^m(T',T')=0$ for all $m\ge1$.
Thus $T'\cong (A/I)^{\oplus r}$ for some $r\ge1$.
Since $\End_A(A/I)\cong A/I$ is local and $T'$ is indecomposable, we have $r=1$.
\end{proof}

We consider a tilting module with an indecomposable direct summand supported at a single vertex.
After relabelling, we may assume that this vertex is $1$.

\begin{setup}\label{setup:bottom}
Put $\varepsilon=1-e_1$, $B=\varepsilon A\varepsilon$, and $D_1=A/A\varepsilon A$.
Let $T=\bigoplus_{i=1}^nT(i)$ be a labelled decomposition of a tilting $A$-module such that $T(1)\cong D_1$.
Put $\Gamma=\{2,\dots,n\}$, $\PP=\PP_\Gamma$, $\QQ=\QQ_\Gamma$, and $L=-\otimes_B\varepsilon A$.
Finally, let  $U_i=T(i)\varepsilon A$ be the trace of $\varepsilon A$ in $T(i)$ and $U=\bigoplus_iU_i$.
\end{setup}

By Proposition~\ref{prop:bottom}, the assumption of Setup~\ref{setup:bottom} holds whenever $\supp T(1)=\{1\}$.

\begin{lemma}\label{lem:free}
In Setup~\ref{setup:bottom}, the following statements hold.
\begin{enumerate}[label=\textup{(\arabic*)}]
\item For every $i$, the module $T(i)/U_i$ is a free $D_1$-module, or equivalently, $T(i)/U_i\in\add T(1)$.
\item $U_i=\rej_{T(1)}T(i)$ for every $i$.
\item For every $D_1$-module $Y$, we have $\Ext_A^j(U,Y)=0$ for all $j\ge0$. In particular, $U\in\PP$.
\end{enumerate}
\end{lemma}

\begin{proof}
Put $I=A\varepsilon A$.
We first show that
\begin{equation}\label{eq:free-ext}
\Ext_A^m(T,Y)=0
\end{equation}
for every $D_1$-module $Y$ and $m\ge1$.
Take a projective resolution $\cdots\to F_1\xrightarrow{d_1} F_0\xrightarrow{d_0} Y\to0$ of $Y$ over $D_1$.
Put $Y_0=Y$ and $Y_{l+1}=\Ker d_l$ for $l\ge0$.
Since $F_l\in\add T(1)\subseteq\add T$, condition \textup{(T2)} gives $\Ext_A^m(T,F_l)=0$ for all $m\ge1$.
Hence, applying $\Hom_{A}(T,-)$ to $0\to Y_{l+1}\to F_l\to Y_l\to0$ yields isomorphisms 
\[
\Ext_A^m(T,Y)
\cong\Ext_A^{m+1}(T,Y_1)
\cong\cdots
\cong\Ext_A^{m+r}(T,Y_r)
\]
for every $m\ge1$ and $r\ge1$. 
If we take $r$ such that $m+r>\pd_A T$, then the last term vanishes by \textup{(T1)}.
This proves \eqref{eq:free-ext}.

\textup{(1)}
Put $V_i=T(i)/U_i$.
Since $U_i=T(i)\varepsilon A$, we have $V_i\varepsilon=0$, and hence $V_i$ is a $D_1$-module.
Moreover, for every $D_1$-module $Y$, we obtain $\Hom_A(U_i,Y)=0$ by $U_i=T(i)\varepsilon A$ and $Y\varepsilon=0$.
Hence, applying $\Hom_A(-,Y)$ to $0\to U_i\to T(i)\to V_i\to0$ gives an exact sequence 
\[0=\Hom_A(U_i, Y) \to \Ext_A^1(V_i, Y) \to \Ext_A^1(T(i), Y),\] 
where the last term vanishes by \eqref{eq:free-ext}. 
Thus we obtain $\Ext_A^1(V_i,Y)=0$.
Since $I^2=I$, we have an isomorphism $\Ext_{D_1}^1(V_i,Y)\cong\Ext_A^1(V_i,Y)=0$ for every $D_1$-module $Y$. 
Thus $V_i$ is a projective $D_1$-module.
Since $D_1$ is local, $V_i$ is free, and hence $V_i\in\add T(1)$.

\textup{(2)}
Since $T(1)\varepsilon=0$, every homomorphism $T(i)\to T(1)$ vanishes on
$U_i=T(i)\varepsilon A$. Hence $U_i\subseteq\rej_{T(1)}T(i)$.
Conversely, by (1), there is an isomorphism $V_i\cong T(1)^{\oplus r_i}$ for some
$r_i\ge0$. Let $\pi\colon T(i)\to T(1)^{\oplus r_i}$ be the composite of this isomorphism with the canonical epimorphism $T(i)\to V_i$. Then $\Ker\pi=U_i$, and
Lemma~\ref{lem:rej-approx}(1) yields $\rej_{T(1)}T(i)\subseteq U_i$.

\textup{(3)}
Let $Y$ be a $D_1$-module.
By $U=\oplus_{i=1}^n U_i$ and $Y\varepsilon=0$, we have $\Hom_A(U,Y)=0$.
Applying $\Hom_A(-,Y)$ to an exact sequence $0\to U\to T\to T/U\to0$ gives an exact sequence 
\[
\Ext_A^j(T,Y)\to\Ext_A^j(U,Y)\to\Ext_A^{j+1}(T/U,Y)
\]
for every $j\ge1$.
Since it follows from (1) that $T/U\in\add T(1)$, the outer terms vanish by \eqref{eq:free-ext}. 
Hence $\Ext_A^j(U,Y)=0$ for all $j\ge1$.
Since $\Hom_A(U,Y)=0$, the assertion holds. 
In particular, $U\in\PP$ by Lemma~\ref{lem:APT}(1).
\end{proof}

The following proposition shows that $T\varepsilon$ is a tilting $B$-module.

\begin{proposition}\label{prop:truncation}
In Setup~\ref{setup:bottom}, the following statements hold.
\begin{enumerate}[label=\textup{(\arabic*)}]
\item $T\varepsilon=U\varepsilon$ is a tilting $B$-module.
\item $T(1)\varepsilon=0$, and the modules $T(i)\varepsilon$ with $i\ge2$ are nonzero and indecomposable.
\item For $i\ge2$, the support of the $B$-module $T(i)\varepsilon$ with respect to the idempotents $e_2,\dots,e_n$ of $B$ coincides with $\supp T(i)\setminus\{1\}$.
\end{enumerate}
\end{proposition}

\begin{proof}
\textup{(1)}
By Lemma~\ref{lem:free}(1), we have
$T/U\in\add T(1)\subseteq\add T$.
Since $T(1)\varepsilon=0$, we obtain
$T\varepsilon=U\varepsilon$.
We show \textup{(T1)}.
Since both $T$ and $T/U$ have finite projective dimension, so does $U$.
By Lemma~\ref{lem:free}(3), $U\in\PP$, and hence all terms of the minimal projective resolution of $U$ are contained in $\add\varepsilon A$.
Applying $(-)\varepsilon$ to the minimal projective resolution gives a finite projective resolution of $U\varepsilon$ over $B$.
Thus \textup{(T1)} holds.

Next, we show \textup{(T2)}.
Applying $\Hom_A(-,T)$ to $0\to U\to T\to T/U\to0$ gives an exact sequence 
\[
\Ext_A^j(T,T)\to
\Ext_A^j(U,T)\to
\Ext_A^{j+1}(T/U,T)
\]
for every $j\ge1$.
Since $T/U\in\add T$ and $T$ satisfies \textup{(T2)},
the outer terms vanish.
Hence $\Ext_A^j(U,T)=0$ for all $j\ge1$.
Since $T/U$ is a $D_1$-module, $\Ext_A^{j}(U,T/U)=0$ for all $j\geq0$ by Lemma~\ref{lem:free}(3).
Hence, applying $\Hom_A(U,-)$ to $0\to U\to T\to T/U\to0$ gives an isomorphism $\Ext_A^j(U,U)\cong \Ext_A^j(U,T)=0$ for all $j\ge1$. 
Since $U\in\PP$, it follows from Lemma~\ref{lem:APT}(2) that
\[
\Ext_B^j(U\varepsilon,U\varepsilon)
\cong\Ext_A^j(U,U)=0
\]
for all $j\ge1$.
Hence \textup{(T2)} holds.

Finally, we show \textup{(T3)}.
By Proposition~\ref{prop:tilting}(1) and \textup{(T3)} for $T$, we have $\varepsilon A\in\cores(\add T)$.
Thus there is an exact sequence $0\to\varepsilon A\to T^0\to\cdots\to T^m\to0$ with $T^l\in\add T$.
Applying $(-)\varepsilon$ yields an exact sequence 
\[
0\to B\to T^0\varepsilon\to\cdots
\to T^m\varepsilon\to0,
\]
where $T^l\varepsilon\in\add T\varepsilon$.
Thus \textup{(T3)} holds, and hence $T\varepsilon=U\varepsilon$ is a tilting $B$-module.

\textup{(2)}
Since $T(1)\cong D_1=A/A\varepsilon A$, we have $T(1)\varepsilon=0$. 
Let $i\ge2$.
Suppose $T(i)\varepsilon=0$.
Then $U_i=T(i)\varepsilon A=0$, and hence $T(i)=V_i\in\add T(1)$ by Lemma~\ref{lem:free}(1).
Since $T(i)$ is indecomposable, we obtain $T(i)\cong T(1)$.
Then $T$ has at most $n-1$ isomorphism classes of indecomposable direct summands, a contradiction.
Thus $T(i)\varepsilon\ne0$.

Next, we show that the natural ring homomorphism $\rho_i\colon\End_A(T(i))\to\End_B(T(i)\varepsilon)$ induced by $(-)\varepsilon$ is surjective.
Put $V_i=T(i)/U_i$.
By Lemma~\ref{lem:free}(1), $V_i\in\add T(1)$.
Hence $\Ext_A^1(V_i,T(i))=0$ by \textup{(T2)}.
Applying $\Hom_A(-,T(i))$ to $0\to U_i\to T(i)\to V_i\to0$ yields a surjection
\[
\End_A(T(i))
\twoheadrightarrow
\Hom_A(U_i,T(i)).
\]
Since $U_i$ is a direct summand of $U$, we have $U_i\in\PP$ by Lemmas~\ref{lem:APT}(1) and~\ref{lem:free}(3).
By $U_i\varepsilon=T(i)\varepsilon$ and Lemma~\ref{lem:APT}(2), we have an isomorphism
\[
\Hom_A(U_i,T(i))
\xrightarrow{\sim}
\End_B(T(i)\varepsilon).
\]
Since the composite of these maps is $\rho_i$, it is surjective.
Since $T(i)$ is indecomposable, $\End_A(T(i))$ is local, and hence so is $\End_B(T(i)\varepsilon)$. 
Thus $T(i)\varepsilon$ is indecomposable.

\textup{(3)}
For $j\ge2$, we have $\varepsilon e_j=e_j$, and hence $(T(i)\varepsilon)e_j=T(i)e_j$. 
Thus $\supp_B(T(i)\varepsilon)=\supp_A T(i)\setminus\{1\}$.
\end{proof}

The following proposition allows us to complete the induction step in the proof of Theorem~\ref{thm:main}.

\begin{proposition}\label{prop:lifting}
In Setup~\ref{setup:bottom}, assume in addition that $B$ is standardly stratified with respect to $2<\dots<n$ and that $T(i)\varepsilon\cong T^B_\le(i)$ for all $i\ge2$, where $T^B_\le$ is the characteristic tilting $B$-module. Then the labelled decomposition $T=\bigoplus_iT(i)$ is reject-triangular. Consequently, $(A,\le)$ is standardly stratified for the natural order and $T(i)\cong T_\le(i)$ for all $i$.
\end{proposition}

\begin{proof}
Since $\Theta_T(1)=T(1)\cong D_1$, we have $\Top\Theta_T(1)\cong S(1)$ and $C_T(1)=0$.
For $i\ge2$, put $U_{<i}=\bigoplus_{j<i}U_j$.

\emph{Step 1: $\Theta_T(i)=\rej_{U_{<i}}U_i$.}
Put $V_i=T(i)/U_i$.
Since $U_i=U_i\varepsilon A$ and $T_{<i}\varepsilon A=U_{<i}$, the inclusion $\eta: U_{<i}\hookrightarrow T_{<i}$ induces an isomorphism
\[\Hom_A(U_i, \eta)\colon
\Hom_A(U_i,U_{<i})
\rightarrow\Hom_A(U_i,T_{<i}).
\]
By Lemma~\ref{lem:free}(1) and \textup{(T2)}, we have $\Ext_A^1(V_i,T_{<i})=0$.
Hence, applying $\Hom_A(-,T_{<i})$ to $0\to U_i\xrightarrow{\iota} T(i)\to V_i\to0$ gives an exact sequence
\[
0\to\Hom_A(V_i,T_{<i})
\to\Hom_A(T(i),T_{<i})
\xrightarrow{\Hom_{A}(\iota,T_{<i})}
\Hom_A(U_i,T_{<i})\to0.
\]
Let $\varphi:=\Hom_{A}(U_i,\eta)^{-1}\Hom_{A}(\iota, T_{<i})\colon \Hom_A(T(i),T_{<i})\to \Hom_A(U_i,U_{<i})$. 
For every $\alpha\in\Hom_A(T(i),T_{<i})$, we have $\eta\varphi(\alpha)=\alpha\iota$.
Since $\eta$ is a monomorphism, we have $\Ker \varphi(\alpha)=\Ker\alpha\cap U_i.$ 
Since $\varphi$ is surjective, we obtain
\[
\rej_{U_{<i}}U_i
=
\bigcap_{\alpha\in\Hom_A(T(i),T_{<i})}
(\Ker\alpha\cap U_i)
=
\Theta_T(i)\cap U_i.
\]
Since $T(1)$ is a direct summand of $T_{<i}$, it follows from Lemma~\ref{lem:free}(2) that
\[
\Theta_T(i)
=\rej_{T_{<i}}T(i)
\subseteq\rej_{T(1)}T(i)=U_i.
\]
Hence $\Theta_T(i)=\rej_{U_{<i}}U_i$ holds. 

\emph{Step 2: $\Theta_T(i)\cong L(\Delta_B(i))$ and $U_i/\Theta_T(i)\in\F(\Theta_T(2),\dots,\Theta_T(i-1))$.}
Let $g_i\colon U_i\to W_i$ be a minimal left $\add U_{<i}$-approximation.
By Lemma~\ref{lem:rej-approx}(2) and Step~1, we have $\Ker g_i=\rej_{U_{<i}}U_i=\Theta_T(i)$.
We show that $g_i\varepsilon$ is a minimal left $\add(U_{<i}\varepsilon)$-approximation.
By Lemma~\ref{lem:free}(3), $U\in\PP$, and hence we have $U_i,W_i\in\PP$.
By Lemma~\ref{lem:APT}(2), the functor $(-)\varepsilon$ is fully faithful on $\PP$.
Hence, for every homomorphism $f\colon U_i\varepsilon\to(U_{<i}\varepsilon)^{\oplus r}$, there is $\widetilde f\colon U_i\to U_{<i}^{\oplus r}$ such that $f=\widetilde f \varepsilon$. 
Since $g_i$ is a left $\add U_{<i}$-approximation, there exists $a\colon W_i\to U_{<i}^{\oplus r}$ such that $\widetilde f=ag_i$.
Thus $f=(a\varepsilon)(g_i\varepsilon).$
To show minimality, let $h\in\End_B(W_i\varepsilon)$ satisfy $h(g_i\varepsilon)=g_i\varepsilon$.
Since $(-)\varepsilon$ is fully faithful on $\PP$, there is a homomorphism $\widetilde h\in\End_A(W_i)$ such that $h=\widetilde h\varepsilon$.
Hence we have $\widetilde h g_i=g_i$.
The left minimality of $g_i$ implies that
$\widetilde h$ is an automorphism, and hence so is $h$.

By the assumption, $U_i\varepsilon\cong T^B_{\le}(i)$ and $U_{<i}\varepsilon=T_{<i}\varepsilon\cong\bigoplus_{2\le j<i}T^B_{\le}(j)$. Let $p\colon U_i\varepsilon\twoheadrightarrow N_i$ and $v\colon N_i\hookrightarrow W_i\varepsilon$ be the canonical maps, where $N_i=\Img(g_i\varepsilon)$, so that $g_i\varepsilon=vp$. By Lemma~\ref{lem:chtilt-rej}(2), we have $\Ker p\cong\Delta_B(i)$ and $N_i,\Cok v\in\F(\Delta_B(<i))$.
Since $U\in\PP$, we have $T^B_{\le}\cong T\varepsilon=U\varepsilon\in\QQ$ by Lemma~\ref{lem:APT}(2). As $\QQ$ is closed under kernels of epimorphisms, Proposition~\ref{prop:orthogonal} gives $\F(\Delta_B)=\cores(\add T^B_{\le})\subseteq\QQ$.
By Lemma~\ref{lem:APT}(2), the counit map $\mu_X\colon L(X\varepsilon)\to X$ is an isomorphism for $X\in\PP$. Since $U_i,W_i\in\PP$, we have $g_i=\mu_{W_i}L(v)L(p)\mu_{U_i}^{-1}$ by the naturality of $\mu$. Here $L(p)$ is an epimorphism and $L(v)$ is a monomorphism since $L$ is right exact and $\Tor_1^B(\Cok v,\varepsilon A)=0$, respectively. Hence $\Img g_i\cong L(N_i)$. Since $\Tor_1^B(N_i,\varepsilon A)=0$, we have $\Theta_T(i)=\Ker g_i\cong\Ker L(p)\cong L(\Delta_B(i))$.
Since $N_i\in\F(\Delta_B(<i))\subseteq\QQ$, $L$ is exact on $\QQ$, and $L(\Delta_B(j))\cong\Theta_T(j)$ for $2\le j<i$, we obtain \[U_i/\Theta_T(i)\cong\Img g_i\cong L(N_i)\in\F(\Theta_T(2),\dots,\Theta_T(i-1)).\]

\emph{Step 3: \textup{(C1)} and \textup{(C2)}.}
Let $i\ge2$. Since $L$ is right exact and $L(e_iB)\cong e_iA=P(i)$, the module $\Theta_T(i)\cong L(\Delta_B(i))$ is a factor module of $P(i)$. Moreover, it is nonzero since $L(\Delta_B(i))\varepsilon\cong\Delta_B(i)\ne0$. Hence $\Top\Theta_T(i)\cong S(i)$, and \textup{(C1)} holds.

Since $\Theta_T(i)\subseteq U_i\subseteq T(i)$,
there is an exact sequence
\[
0\to U_i/\Theta_T(i)\to
C_T(i)\to T(i)/U_i=V_i \to0.
\]
By Step~2, we have
$U_i/\Theta_T(i)\in\F(\Theta_T(2),\dots,\Theta_T(i-1))$. 
By Lemma~\ref{lem:free}(1), $V_i \in \add T(1)=\add\Theta_T(1)$.
Since $\F(\Theta_T(<i))$ is closed under extensions, we obtain $C_T(i)\in\F(\Theta_T(<i))$. 
Thus \textup{(C2)} holds.

Therefore the labelled decomposition is reject-triangular.
Since $T$ is a tilting module, the assertions follow from Theorem~\ref{thm:reject}.
\end{proof}

We prove Theorem~\ref{thm:main}.

\begin{proof}[Proof of Theorem~\ref{thm:main}]
If \textup{(i)} holds, then \textup{(ii)} follows from Theorem~\ref{thm:chtilt}(2).
We prove the converse. 

\emph{Step 1: the case of a total order.}
Assume that $\unlhd$ is a total order.
After relabelling, we may assume that ${\unlhd}={\le}$ is the natural order.
We proceed by induction on $n$.
If $n=1$, then $T=T(1)\cong A$ by Proposition~\ref{prop:bottom}.
Since $K(1)=0$ and $\Delta(1)=A$, the algebra $A$ is standardly stratified and its characteristic tilting module is $A$.

Let $n\ge2$ and assume \textup{(ii)}.
Since $T(1)$ is nonzero and
$\supp T(1)\subseteq\{1\}$, we have
$\supp T(1)=\{1\}$.
Hence it follows from Proposition~\ref{prop:bottom} that $T(1)\cong A/A(1-e_1)A=D_1$.
Put $\varepsilon=1-e_1$ and $B=\varepsilon A\varepsilon$, as in Setup~\ref{setup:bottom}.
By Proposition~\ref{prop:truncation}(1) and (2), $T\varepsilon=\bigoplus_{i=2}^n T(i)\varepsilon$ is a tilting $B$-module, and the modules $T(i)\varepsilon$ for $i\ge2$ are nonzero and indecomposable. 
Moreover, by Proposition~\ref{prop:truncation}(3), 
\[
\supp_B(T(i)\varepsilon)
=\supp_A T(i)\setminus\{1\}
\subseteq\{2,\dots,i\}
\]
for every $i\ge2$.
Thus this decomposition is triangular with respect to the order $2<\dots<n$.
By the induction hypothesis, $(B,\le)$ is standardly stratified and $T(i)\varepsilon\cong T^B_\le(i)$ for every $i\ge2$.
Thus it follows from Proposition~\ref{prop:lifting} that $(A,\le)$ is standardly stratified and $T(i)\cong T_\le(i)$ for every $i$.

\emph{Step 2: the general case.} 
Let $\unlhd$ be an arbitrary partial order satisfying \textup{(ii)}, and take a linear extension $\le$ of $\unlhd$.
Then
\[
\supp T(i)
\subseteq\{j\mid j\unlhd i\}
\subseteq\{j\mid j\le i\}
\]
for every $i$.
Hence $(A,\le)$ is standardly stratified and $T(i)\cong T_\le(i)$ for every $i$ by Step~1.
By Proposition~\ref{prop:descent}, $(A,\unlhd)$ is standardly stratified and $T_\unlhd(i)\cong T_\le(i)$ for every $i$.
Thus $T(i)\cong T_\unlhd(i)$.

\emph{Step 3: the statements (a) and (b).}
We may assume that $T(i)=T_\unlhd(i)$ for every $i$. Let
$0\to\Delta(i)\xrightarrow{\beta_i}T(i)\to X(i)\to0$ be the exact sequence in
Theorem~\ref{thm:chtilt}(2). Since $X(i)\in\F(\Delta(\lhd i))$ and
$\supp\Delta(j)\subseteq\{l\mid l\unlhd j\}$ for every $j$, we have $X(i)e_i=0$.
Applying $(-)e_i$, we obtain $\beta_i(\Delta(i)e_i)=T(i)e_i$.
For every $X\in\mod A$, each homomorphism $h\colon P(i)=e_iA\to X$ satisfies
$\Img h=h(e_i)A$ with $h(e_i)\in Xe_i$, and hence $\tr_{P(i)}(X)=Xe_iA$.
Since $\Delta(i)$ is a factor module of $P(i)=e_iA$, we have $\Delta(i)=\Delta(i)e_iA$.  
Thus
\[
\beta_i(\Delta(i))=\beta_i(\Delta(i)e_i)A=T(i)e_iA=\tr_{P(i)}(T(i)),
\]
which proves (a). 

By Lemma~\ref{lem:Delta}(1), $T(i)e_i=\beta_i(\Delta(i)e_i)\cong\Delta(i)e_i\cong D_i$
as right $D_i$-modules. Since $X(i)e_i=0$, we have
$[X(i):S(i)]=0$, and hence $[T(i):S(i)]=[\Delta(i):S(i)]=c_i$ by
Lemma~\ref{lem:Delta}(1). Thus (b) holds.
\end{proof}

\section{Standardly stratified structures and triangular tilting modules}\label{sec:bijection}

In this section, we establish an isomorphism of posets between standardly stratified structures and triangular tilting modules. 

\subsection{Standardly stratified structures}
We introduce the notion of standardly stratified structures.

\begin{definition}\label{def:sss}
Two partial orders $\unlhd_1$ and $\unlhd_2$ on $\Lambda$ with respect to which $A$ is standardly stratified are said to be \emph{equivalent} if $\Delta^{\unlhd_1}(i)\cong\Delta^{\unlhd_2}(i)$ for every $i$. A \emph{standardly stratified structure} on $A$ is an equivalence class $[\unlhd]$, and $\sss A$ denotes the set of standardly stratified structures. For $[\unlhd_1],[\unlhd_2]\in\sss A$, we write $[\unlhd_1]\preceq[\unlhd_2]$ if $\F(\Delta^{\unlhd_1})\supseteq\F(\Delta^{\unlhd_2})$. The subset of quasi-hereditary structures is denoted by $\qhs A$.
\end{definition}

Every standardly stratified structure contains a total order by Lemma~\ref{lem:refine}.
Since $D_i\cong\End_A(\Delta(i))$ by Lemma~\ref{lem:Delta}(1), whether $(A,\unlhd)$ is quasi-hereditary depends only on $[\unlhd]$.

To show that $\preceq$ is a partial order, we use the following proposition.

\begin{proposition}\label{prop:determine}
Let $(A,\unlhd)$ be standardly stratified. Every nonzero factor module of $P(i)$ lying in $\F(\Delta)$ admits an epimorphism onto $\Delta(i)$.
\end{proposition}

\begin{proof}
Let $M\in\F(\Delta)$ be a nonzero factor module of $P(i)$. Then $\Top M\cong S(i)$. 
Let $\le$ be a linear extension of $\unlhd$. 
Since $\Delta(i)=\Delta^{\le}(i)$ for each $i\in\Lambda$  by Lemma~\ref{lem:refine}, we have $M\in\F(\Delta^{\le})$. 
By Lemma~\ref{lem:Delta}, $\Delta^{\le}$ is a standardizable set with $\Top\Delta^{\le}(i)\cong S(i)$. 
Hence there exists an epimorphism $M\twoheadrightarrow\Delta^{\le}(i)=\Delta(i)$ by Lemma~\ref{lem:top}.
\end{proof}

The relation $\preceq$ is reflexive and transitive.
To see antisymmetry, assume that $[\unlhd_1]\preceq[\unlhd_2]$ and $[\unlhd_2]\preceq[\unlhd_1]$. 
Then $\F(\Delta^{\unlhd_1})=\F(\Delta^{\unlhd_2})$.
By Proposition~\ref{prop:determine}, there are epimorphisms $\Delta^{\unlhd_1}(i)\twoheadrightarrow\Delta^{\unlhd_2}(i)$ and $\Delta^{\unlhd_2}(i)\twoheadrightarrow\Delta^{\unlhd_1}(i)$ for every $i$. 
Since these modules have finite length, they are isomorphic.
Hence $[\unlhd_1]=[\unlhd_2]$, and the relation $\preceq$ is antisymmetric. 
Thus $\sss A$ is a poset.
Its restriction to $\qhs A$ is the order of \cite{FKR}.

The posets $\qhs A$ and $\qhs(A^{\op})$ are anti-isomorphic (see \cite{CPS1} and \cite[Appendix, Statement~9]{DR89}).
In contrast, the property of being standardly stratified is not left-right symmetric in general (see the remark following \cite[Definition~2.2.5]{CPS2}). 
In fact, $\sss A$ and $\sss(A^{\op})$ may even have different cardinalities, as the following example shows.

\begin{example}\label{ex:sss-asymmetry}
Let $A$ and $M$ be as in Example~\ref{ex:ss}.
For right $A$-modules, the orders $1<2$ and $2<1$ give distinct standardly stratified structures. 
The former has
standard modules $M$ and $P(2)$ and the latter has projective standard modules.
For left $A$-modules, only $1<2$ gives a standardly stratified
structure, with standard modules $Ae_1$ and $Ae_2$.
Since every standardly stratified structure contains a total order
by Lemma~\ref{lem:refine}, we obtain $|\sss A|=2$ and $|\sss(A^{\mathrm{op}})|=1$.
\end{example}

The following lemma characterizes the order on $\sss A$ in terms of characteristic tilting modules.

\begin{lemma}\label{lem:order}
Let $(A,\unlhd_1)$ and $(A,\unlhd_2)$ be standardly stratified.
Then the following conditions are equivalent.
\begin{enumerate}[label=\textup{(\roman*)}]
\item $[\unlhd_1]\preceq[\unlhd_2]$.
\item $T_{\unlhd_1}\le T_{\unlhd_2}$.
\item $\Y(\Delta^{\unlhd_1})\subseteq\Y(\Delta^{\unlhd_2})$.
\end{enumerate}
In particular, if $T_{\unlhd_1}\cong T_{\unlhd_2}$, then $[\unlhd_1]=[\unlhd_2]$ holds.
\end{lemma}

\begin{proof}
Since $\Y(\Delta^{\unlhd_l})=\F(\Delta^{\unlhd_l})^{\perp_1}$, $[\unlhd_1]\preceq[\unlhd_2]$ implies $\Y(\Delta^{\unlhd_1})\subseteq\Y(\Delta^{\unlhd_2})$. 
Conversely, if $\Y(\Delta^{\unlhd_1})\subseteq\Y(\Delta^{\unlhd_2})$, then Proposition~\ref{prop:orthogonal} gives $\F(\Delta^{\unlhd_1})={}^\perp\Y(\Delta^{\unlhd_1})\supseteq{}^\perp\Y(\Delta^{\unlhd_2})=\F(\Delta^{\unlhd_2}).$
By Proposition~\ref{prop:tilting}(2), $T_{\unlhd_1}\le T_{\unlhd_2}$ if and only if $T_{\unlhd_1}^{\perp}\subseteq T_{\unlhd_2}^{\perp}$. 
Since $T_{\unlhd_l}^{\perp}=\Y(\Delta^{\unlhd_l})$ by Proposition~\ref{prop:orthogonal}, the remaining equivalence follows.
Finally, the last assertion follows from the equivalence of \textup{(i)} and \textup{(ii)}, since $\preceq$ is antisymmetric.
\end{proof}

\begin{remark}\label{rem:sss-qhs}
If $\qhs A\ne\emptyset$, then $\sss A=\qhs A$.
Indeed, the existence of a quasi-hereditary structure implies
$\gldim A<\infty$, and the assertion follows
from Corollary~\ref{cor:gldim}.
Equivalently, if $A$ admits a standardly stratified structure
which is not quasi-hereditary, then $\qhs A=\emptyset$.
\end{remark}

\subsection{The poset isomorphism}
We first show how a triangular tilting module determines the labelling of its indecomposable summands and the corresponding standard modules.

\begin{proposition}\label{cor:canonical}
Let $T$ be a triangular tilting $A$-module. 
Suppose that the labelled decompositions $T\cong\bigoplus_{i\in\Lambda}T(i)$ and $T\cong\bigoplus_{i\in\Lambda}T'(i)$ are triangular with respect to partial orders $\unlhd_1$ and $\unlhd_2$, respectively. 
Then $T(i)\cong T'(i)$ for every $i\in\Lambda$.
Moreover, $\Delta^{\unlhd_1}(i) \cong T(i)e_iA \cong T'(i)e_iA\cong\Delta^{\unlhd_2}(i)$ for every $i\in\Lambda$.
\end{proposition}

\begin{proof}
Since $T$ is basic, there is a permutation $\sigma$ of $\Lambda$ such that
$T'(i)\cong T(\sigma(i))$ for every $i\in\Lambda$.
By Theorem~\ref{thm:main}(b) and Lemma~\ref{lem:Delta}(1), we have
$[T(j):S(j)]=[\Delta^{\unlhd_1}(j):S(j)]\ne0$, and hence $j\in\supp T(j)$ for every
$j\in\Lambda$. It follows that
\[
\sigma(i)\in\supp T(\sigma(i))=\supp T'(i)\subseteq\{j\mid j\unlhd_2 i\},
\]
that is, $\sigma(i)\unlhd_2 i$ for every $i\in\Lambda$.
Fix $i\in\Lambda$, and let $m\ge1$ be such that $\sigma^m(i)=i$. Then
\[
i=\sigma^m(i)\unlhd_2\sigma^{m-1}(i)\unlhd_2\cdots\unlhd_2\sigma(i)\unlhd_2 i,
\]
and hence $\sigma(i)=i$. Thus $T(i)\cong T'(i)$ for every $i\in\Lambda$.
The remaining isomorphisms follow from Theorem~\ref{thm:main}(a).
\end{proof}

By Proposition~\ref{cor:canonical}, for a triangular tilting module $T$,
the isomorphism class of the summand labelled by each $i\in\Lambda$
is independent of the triangular labelled decomposition and the
partial order.
We introduce the following terminology and notation.

\begin{definition}\label{def:canonical}
Let $T$ be a triangular tilting $A$-module.
\begin{enumerate}[label=\textup{(\arabic*)}]
\item A partial order $\unlhd$ on $\Lambda$ is called
\emph{compatible} with $T$ if $T$ admits a labelled decomposition
that is triangular with respect to $\unlhd$.
\item We define $\preceq_T$ to be the reflexive and transitive
closure of the relation $\{(j,i)\in\Lambda\times\Lambda\mid j\in\supp T(i)\}$, where $T=\bigoplus_{i\in\Lambda}T(i)$ is a triangular labelled decomposition.
\end{enumerate}
\end{definition}

By Proposition~\ref{cor:canonical}, the relation $\preceq_T$ is independent of the choice of the triangular labelled decomposition.

The following proposition shows that the partial orders compatible with a triangular tilting module $T$ are precisely the partial orders for which $T$ is the characteristic tilting module.

\begin{proposition}\label{prop:essential}
Let $T$ be a triangular tilting $A$-module. Then $\preceq_T$ is a partial order, and for
any partial order $\unlhd$ on $\Lambda$, the following conditions are equivalent.
\begin{enumerate}[label=\textup{(\roman*)}]
\item $(A,\unlhd)$ is standardly stratified and $T_\unlhd\cong T$.
\item $\unlhd$ is compatible with $T$.
\item $\unlhd$ contains $\preceq_T$.
\end{enumerate}
In particular, $\preceq_T$ is the minimum partial order compatible with $T$.
\end{proposition}

\begin{proof}
We show that $\preceq_T$ is a partial order. 
Take a labelled decomposition $T=\bigoplus_{i\in\Lambda}T(i)$ that is triangular with respect to some partial order $\unlhd_0$.
Then $j\in\supp T(i)$ implies $j\unlhd_0 i$.
Since $\unlhd_0$ is reflexive and transitive, we have ${\preceq_T}\subseteq{\unlhd_0}$.
In particular, $\preceq_T$ is antisymmetric, and hence it is a partial order.

By Proposition~\ref{cor:canonical}, any triangular labelled decomposition
$T\cong\bigoplus_{i\in\Lambda}T'(i)$ satisfies $T'(i)\cong T(i)$ for every $i\in\Lambda$.
Hence $\unlhd$ is compatible with $T$ if and only if $j\unlhd i$ whenever $j\in\supp T(i)$.
Since $\unlhd$ is reflexive and transitive, this is equivalent to ${\preceq_T}\subseteq{\unlhd}$ by Definition~\ref{def:canonical}(2).
Thus \textup{(ii)} and \textup{(iii)} are equivalent.

If \textup{(ii)} holds, then \textup{(i)} holds by Theorem~\ref{thm:main}. Conversely, suppose that \textup{(i)}
holds. By Theorem~\ref{thm:chtilt}(2), the labelled decomposition
$T_\unlhd=\bigoplus_iT_\unlhd(i)$ is triangular with respect to $\unlhd$. 
Since $T_\unlhd\cong T$, the order $\unlhd$ is compatible with $T$. Thus \textup{(ii)} holds.
\end{proof}

Let $\Ttilt A$ and $\UTtilt A$ denote the sets of isomorphism classes of basic
triangular and unitriangular tilting $A$-modules, respectively. They are regarded as subposets of $\tilt A$.

The following theorem extends the poset correspondence for quasi-hereditary structures in \cite[Theorem~3.22]{ACKT} to standardly stratified structures over arbitrary fields.

\begin{theorem}\label{thm:bijection}
Let $A$ be a basic finite dimensional algebra over a field $k$.
Then there exist mutually inverse isomorphisms of posets
\[
\xymatrix@C=3.5em{
\Ttilt A \ar@<0.5ex>[r]^{\varphi}
&
\sss A \ar@<0.5ex>[l]^{\psi},
}
\]
where $\varphi(T)=[\preceq_T]$ and
$\psi([\unlhd])=T_\unlhd$.
Moreover, these maps restrict to mutually inverse isomorphisms of posets between $\UTtilt A$ and $\qhs A$.
\end{theorem}

\begin{proof}
We show that $\varphi$ is well-defined.
By Proposition~\ref{cor:canonical}, $\preceq_T$ depends only on the isomorphism class of $T$.
Since $(A,\preceq_T)$ is standardly stratified by Proposition~\ref{prop:essential}, $\varphi$ is well-defined.
We show that $\psi$ is well-defined.
Assume $[\unlhd_1]=[\unlhd_2]$.
Then $\Delta^{\unlhd_1}(i)\cong\Delta^{\unlhd_2}(i)$ for all $i\in\Lambda$.
By Theorem~\ref{thm:chtilt}(2), $T_{\unlhd_1}(i)\cong T_{\unlhd_2}(i)$ for all $i\in\Lambda$, and hence $T_{\unlhd_1}\cong T_{\unlhd_2}$. 
Since every characteristic tilting module is triangular by Theorem~\ref{thm:main}, $\psi$ is well-defined.
 
By Proposition~\ref{prop:essential}, we have
$T_{\preceq_T}\cong T$, and hence $\psi(\varphi(T))=T$.
Conversely, let $[\unlhd]\in\sss A$ and put $T=T_\unlhd$. Since $T_{\preceq_T}\cong T$
by Proposition~\ref{prop:essential}, Lemma~\ref{lem:order} yields $[\preceq_T]=[\unlhd]$.
Thus $\varphi(\psi([\unlhd]))=[\unlhd]$.
Let $T,U\in\Ttilt A$.
By Lemma~\ref{lem:order}, $\varphi(T)\preceq\varphi(U)$ holds if and only if $T_{\preceq_T}\le T_{\preceq_U}$.
Since $T_{\preceq_T}\cong T$ and $T_{\preceq_U}\cong U$, this is equivalent to $T\le U$.
Thus $\varphi$ and $\psi$ are mutually inverse isomorphisms of posets.

Finally, it follows from Corollary~\ref{cor:QH} and Proposition~\ref{cor:canonical} that $(A,\unlhd)$ is quasi-hereditary if and only if $T_\unlhd$ is unitriangular.
Hence the isomorphisms restrict to $\UTtilt A$ and $\qhs A$.
\end{proof}

By Remark~\ref{rem:sss-qhs} and Theorem~\ref{thm:bijection}, we have $\Ttilt A=\UTtilt A$ whenever $\qhs A\ne\emptyset$.

The minimum order $\preceq_T$ can also be described in terms of standard and proper costandard modules.

\begin{remark}\label{rem:least-order}
Let $[\unlhd]\in\sss A$ and put $T=\psi([\unlhd])$.
By Proposition~\ref{prop:essential} and Theorem~\ref{thm:bijection}, a partial order on $\Lambda$ belongs to $[\unlhd]$ if and only if it is compatible with $T$.
In particular, $\preceq_T$ is the minimum element of $[\unlhd]$ with respect to inclusion of relations.
Thus $\preceq_T$ coincides with the essential order defined in \cite[Section~2.2]{MPW} (see \cite[Lemma~1]{MPW} and \cite[Section~2]{CK}). 
This order is generated by the relations $j\preceq_T i$ for $j\in\supp\Delta(i)\cup\supp\bar\nabla(i)$.
It is also generated by the relations $i\preceq_T j$ for which $\Hom_A(\Delta(i),\Delta(j))\ne0$ or $\Ext^1_A(\Delta(i),\Delta(j))\ne0$ (see \cite[Lemma~5.3]{Kul} and \cite[proof of Lemma~4.4]{CK}).
In the quasi-hereditary case, $\preceq_T$ coincides with the minimal adapted order of \cite[Definition~2.10]{FKR} (see \cite[Proposition~2.9]{FKR}).
\end{remark}

The following corollary gives a characterization for triangularity of the regular module.

\begin{corollary}\label{cor:regular}
Let $\preceq_A$ be the reflexive and transitive closure of the relation $\{(j,i)\in\Lambda\times\Lambda\mid e_iAe_j\ne0\}.$ 
Then $A_A\in\Ttilt A$ if and only if $\preceq_A$ is a partial order.
In this case, $(A,\preceq_A)$ is standardly stratified with $\Delta^{\preceq_A}(i)\cong P(i)$ for every $i$ and $T_{\preceq_A}\cong A_A$, and $[\preceq_A]$ is the maximum element of $\sss A$.
\end{corollary}

\begin{proof}
Suppose that $\preceq_A$ is a partial order.
Since
\[
\supp P(i)=\{j\mid e_iAe_j\ne0\}
\subseteq\{j\mid j\preceq_A i\},
\]
the labelled decomposition $A=\bigoplus_{i\in\Lambda}P(i)$ is triangular.
By Theorem~\ref{thm:main}, $(A,\preceq_A)$ is standardly stratified and $T_{\preceq_A}\cong A_A\in\Ttilt A$.
Moreover, by Theorem~\ref{thm:main}(a), $\Delta^{\preceq_A}(i)\cong P(i)e_iA=P(i).$

Conversely, suppose that $A_A\in\Ttilt A$.
By Theorem~\ref{thm:bijection}, there exists a partial order $\unlhd$ such that $(A,\unlhd)$ is standardly stratified and $T_\unlhd\cong A_A$.
By Proposition~\ref{prop:orthogonal}, $\F(\Delta^\unlhd)=\cores(\add A)=\add A.$
Thus each $\Delta^\unlhd(i)$ is projective, and hence $\Delta^\unlhd(i)\cong P(i)$.
Since $\supp P(i)=\supp\Delta^\unlhd(i)\subseteq\{j\mid j\unlhd i\}$ for every $i$, we have ${\preceq_A}\subseteq{\unlhd}$.
Thus $\preceq_A$ is a partial order.

Since $A_A$ is the maximum element of $\tilt A$, it follows from Theorem~\ref{thm:bijection} that $[\preceq_A]$ is the maximum element of $\sss A$.
\end{proof}

The following example illustrates the poset isomorphism of Theorem~\ref{thm:bijection}. 

\begin{example}\label{ex:ss2}
Let $A$ and $T=M\oplus P(1)$ be as in Example~\ref{ex:ss}.
The labelled decomposition $T(1)=M$, $T(2)=P(1)$ is triangular since $\supp T(1)=\{1\}$ and $\supp T(2)=\{1,2\}$.
Hence $\preceq_T=(1<2)$ by Definition~\ref{def:canonical}(2).
For the regular module $A_A=P(1)\oplus P(2)$, we have $\supp P(1)=\{1,2\}$ and $\supp P(2)=\{2\}$.
Thus $\preceq_A=(2<1)$, and $A_A$ is triangular by
Corollary~\ref{cor:regular}.
Together with Example~\ref{ex:sss-asymmetry},
Theorem~\ref{thm:bijection} yields
\[
\sss A=\{[1<2]\prec[2<1]\},
\qquad
\Ttilt A=\{T<A_A\},
\]
with $\varphi(T)=[1<2]$ and $\varphi(A_A)=[2<1]$.\end{example}

As a consequence of Theorem~\ref{thm:main}, triangularity is preserved under Ringel duality.

\begin{corollary}\label{cor:ringel}
Let $T=\bigoplus_{i\in\Lambda}T(i)$ be a labelled decomposition of a basic tilting $A$-module, and put $\Gamma=\End_A(T)^{\op}$.
For $i\in\Lambda$, let $e^\Gamma_i\in\Gamma$ be the idempotent given by composing the canonical projection $T\to T(i)$ with the canonical inclusion $T(i)\to T$.
We label the indecomposable projective $\Gamma$-module $e^\Gamma_i\Gamma\cong\Hom_A(T(i),T)$ and its top $S_\Gamma(i)$ by $i$.
Then the following statements hold.
\begin{enumerate}[label=\textup{(\arabic*)}]
\item For all $i,j\in\Lambda$, we have $[Te_i:S_\Gamma(j)]_\Gamma\ne0$ if and only if $[T(j):S(i)]_A\ne0$.
\item Let $\unlhd$ be a partial order on $\Lambda$.
Then the labelled decomposition $T=\bigoplus_{i}T(i)$ is triangular with respect to $\unlhd$ if and only if the labelled decomposition $T_\Gamma=\bigoplus_{i}Te_i$ is triangular with respect to $\unlhd^{\op}$.
In this case, $(\Gamma,\unlhd^{\op})$ is standardly stratified, $T_\Gamma$ is the characteristic tilting module of $(\Gamma,\unlhd^{\op})$, and $\preceq_{T_\Gamma}=(\preceq_T)^{\op}$.
\end{enumerate}
\end{corollary}

\begin{proof}
By \cite[Theorem~1.5]{Mi}, $T_\Gamma$ is a tilting $\Gamma$-module, and the natural map $A\to\End_\Gamma(T_\Gamma)^{\op}$ is an algebra isomorphism.
Since the actions of $A$ and $\Gamma$ on $T$ commute, we have 
\[
T_\Gamma\cong\Hom_A(A,T)
\cong\bigoplus_{i\in\Lambda}\Hom_A(P(i),T)\cong\bigoplus_{i\in\Lambda} Te_i
\]
as $\Gamma$-modules.
Since $A\cong\End_\Gamma(T_\Gamma)^{\op}$ is basic, $T_\Gamma=\bigoplus_{i\in\Lambda}Te_i$ is a labelled decomposition of the basic tilting $\Gamma$-module $T_\Gamma$.

\textup{(1)}
We have isomorphisms 
$
(Te_i)e^\Gamma_j
\cong\Hom_A(P(i),T(j))
\cong T(j)e_i
$ as $k$-vector spaces. 
The composition multiplicity $[Te_i:S_\Gamma(j)]_\Gamma$ is nonzero
if and only if $(Te_i)e^\Gamma_j\ne0$.
Similarly, $[T(j):S(i)]_A\ne0$ if and only if $T(j)e_i\ne0$.
Hence \textup{(1)} holds.

\textup{(2)}
By definition, $T=\bigoplus_iT(i)$ is triangular with respect
to $\unlhd$ if and only if $[T(j):S(i)]_A\ne0$ implies
$i\unlhd j$ for all $i,j\in\Lambda$.
Similarly, $T_\Gamma=\bigoplus_iTe_i$ is triangular with respect
to $\unlhd^{\op}$ if and only if
$[Te_i:S_\Gamma(j)]_\Gamma\ne0$ implies
$j\unlhd^{\op}i$, that is, $i\unlhd j$,
for all $i,j\in\Lambda$.
By \textup{(1)}, these two conditions are equivalent.

Suppose that these conditions hold.
Applying Theorem~\ref{thm:main} to $\Gamma$, the labelled
decomposition $T_\Gamma=\bigoplus_iTe_i$ and the partial order
$\unlhd^{\op}$, we see that $(\Gamma,\unlhd^{\op})$ is
standardly stratified and that $T_\Gamma$ is its
characteristic tilting module.
Moreover, by \textup{(1)}, we have
\[
\{(j,i)\mid j\in\supp_\Gamma(Te_i)\}
=
\{(j,i)\mid i\in\supp_A T(j)\}
=
\{(i,j)\mid j\in\supp_A T(i)\}.
\]
Thus the generating relations for $\preceq_{T_\Gamma}$
are the opposites of those for $\preceq_T$.
Taking reflexive and transitive closures, we obtain
$\preceq_{T_\Gamma}=(\preceq_T)^{\op}$
by Definition~\ref{def:canonical}(2).
\end{proof}

\section{Tensor algebras and GLS path algebras}\label{sec:GLS}

We study standardly stratified structures of tensor algebras, including GLS path algebras. 
We apply Theorems~\ref{thm:main} and~\ref{thm:bijection} to describe the corresponding triangular tilting modules.

\subsection{Standardly stratified structures of tensor algebras}\label{subsec:support}
For $s\in\sss A$, we write $\Delta_s$ and $T_s$ for its standard modules and its characteristic tilting module. We first generalize the hereditary case of \cite[Lemma~2.21]{FKR}.

\begin{proposition}\label{prop:support-order}
Let $s,t\in\sss A$, and suppose that $\pd_A\Delta_s(i)\le1$ and $[\Delta_s(i):S(i)]=[\Delta_t(i):S(i)]$ for all $i$. 
Then the following conditions are equivalent. 
\begin{enumerate}[label=\textup{(\roman*)}]
\item $s\preceq t$, that is, $\F(\Delta_s)\supseteq\F(\Delta_t)$.
\item There is an epimorphism $\Delta_t(i)\to\Delta_s(i)$ for every $i$.
\item $\supp\Delta_s(i)\subseteq\supp\Delta_t(i)$ for every $i$.
\end{enumerate}
The implications \textup{(i)}$\Rightarrow$\textup{(ii)}$\Rightarrow$\textup{(iii)} and \textup{(iii)}$\Rightarrow$\textup{(ii)} do not require the assumptions on projective dimension and composition multiplicities.
\end{proposition}

\begin{proof}
If (i) holds, then $\Delta_t(i)\in\F(\Delta_s)$ is a nonzero factor module of $P(i)$. 
By Proposition~\ref{prop:determine}, (ii) holds.
Moreover, (ii) $\Rightarrow$ (iii) is clear. Assume (iii), and let $\unlhd$ be a partial order such that $t=[\unlhd]$. Then $\supp\Delta_s(i)\subseteq\{j\mid j\unlhd i\}$, and since $\Delta_s(i)$ is a factor module of $P(i)$, the maximality of $\Delta_t(i)$ gives an epimorphism $\Delta_t(i)\to\Delta_s(i)$. Thus (iii) implies (ii).

Assume that (ii) holds, $\pd_A\Delta_s(i)\le1$ and $[\Delta_s(i):S(i)]=[\Delta_t(i):S(i)]$ for all $i$.
Take an exact sequence
\begin{equation}\label{eq:support-kernel}
0\longrightarrow K_i\longrightarrow\Delta_t(i)\longrightarrow\Delta_s(i)\longrightarrow0.
\end{equation}
Let $Y\in\Y(\Delta_t)$. 
By Proposition~\ref{prop:orthogonal}, $\Ext_A^m(\Delta_t(i),Y)=0$ for all $m\ge1$. 
Applying $\Hom_A(-,Y)$ to \eqref{eq:support-kernel}  yields an exact sequence 
\[\Ext_A^m(\Delta_t(i),Y)\to \Ext_A^m(K_i,Y) \to \Ext_A^{m+1}(\Delta_s(i),Y)\]
for all $m\ge1$. 
Since $\pd\Delta_s(i)\le1$, we obtain $\Ext_A^m(K_i,Y)=0$ for all $m\ge1$. 
By Proposition~\ref{prop:orthogonal}, $K_i\in{}^\perp\Y(\Delta_t)=\F(\Delta_t).$
Let $\unlhd$ be a partial order such that $t=[\unlhd]$.
Since $[\Delta_s(i):S(i)]=[\Delta_t(i):S(i)]$, we have $[K_i:S(i)]=0$. 
Hence it follows from $\supp K_i\subseteq\{j\mid j\unlhd i\}$ that every $\Delta_t(j)$ occurring in a $\Delta_t$-filtration of $K_i$ satisfies $j\lhd i$. 
We show $\Delta_t(i)\in\F(\Delta_s)$ by induction on $i$ with respect to $\unlhd$.
If $i$ is minimal, then $K_i=0$, and hence $\Delta_t(i)\cong \Delta_s(i)\in \F(\Delta_s)$. 
Assume $\Delta_t(j)\in\F(\Delta_s)$ for all $j\lhd i$. 
Then $K_i\in \F(\Delta_s)$ by the induction hypothesis. 
Since $\F(\Delta_s)$ is closed under extensions, \eqref{eq:support-kernel} gives $\Delta_t(i)\in\F(\Delta_s)$.
Thus $\F(\Delta_t)\subseteq\F(\Delta_s)$.
\end{proof}

Next, we compare standardly stratified structures of acyclic tensor algebras with quasi-hereditary structures of path algebras.

Let $Q$ be a finite acyclic quiver without parallel arrows. 
At vertex $i$, take a nonzero finite dimensional local $k$-algebra $R_i$, and at an arrow $a:i\to j$, take a nonzero finite dimensional $R_i$--$R_j$-bimodule $B_a$ which is free on both sides. Put
\[
S=\prod_{i\in Q_0}R_i,\qquad B=\bigoplus_{a\in Q_1}B_a,
\qquad A=T_S(B).
\]
In this section, an \emph{acyclic tensor algebra}
means an algebra $T_S(B)$ constructed from
$Q$, $(R_i)_{i\in Q_0}$ and $(B_a)_{a\in Q_1}$
satisfying the above assumptions.
For each $i$, let $e_i$ denote the identity of $R_i$, viewed as an idempotent of $S\subseteq A$.
We use the right module convention in which $e_iBe_j=B_a$ for each arrow $a:i\to j$.
We also write $B_{ij}=B_a$ for the arrow $a:i\to j$.
Since $Q$ is acyclic, $A$ is finite dimensional and $e_iAe_i=R_i$.
For a path $p=(i_0\to i_1\to\cdots\to i_r)$, put
\[
B_p=B_{i_0i_1}\otimes_{R_{i_1}}\cdots
\otimes_{R_{i_{r-1}}}B_{i_{r-1}i_r};
\]
for a length-zero path at $i$, put $B_p=R_i$. These bimodules are nonzero and free on both sides. 
As an $S$-bimodule, $A$ is the direct sum of the $B_p$ over all paths $p$ in $Q$, including paths of length zero.
If $p$ ends at $j$ and $q$ starts at $j$, multiplication is given by $B_p\otimes_{R_j}B_q\cong B_{pq}$.
Products of components whose paths cannot be joined are zero.
An $A$-module $M$ is called \emph{locally free} if $Me_i$ is a free right $R_i$-module for every vertex $i$. We call a total order on $Q_0$ \emph{topological} if $i<j$ for every arrow $i\to j$, and \emph{reverse topological} if $j<i$ for every arrow $i\to j$.

We describe the standardly stratified structures on $A$.

\begin{theorem}\label{thm:tensor}
Let $A=T_S(B)$ be an acyclic tensor algebra with quiver $Q$.
Then the following statements hold.
\begin{enumerate}[label=\textup{(\arabic*)}]
\item For every total order $\le$ on $Q_0$, $(A,\le)$ is standardly stratified, and every $M\in \F(\Delta^{\le})$ has projective dimension at most one.
\item For every total order $\le$ and every vertex $i$, we have $D_i=R_i$, and $\Delta(i)$ is locally free with support
\begin{equation}\label{eq:path-support}
\Sigma_i^\le=
\{j\mid \text{there is a path }i=v_0\to\cdots\to v_r=j
\text{ with }v_\ell\le i\text{ for all }\ell\}.
\end{equation}
\item For total orders $\le$ and $\le'$ on $Q_0$, we have $[\le]\preceq[\le']$ in
$\sss A$ if and only if $\Sigma_i^{\le}\subseteq\Sigma_i^{\le'}$ for every $i\in Q_0$.
Consequently, sending the class of a total order for $A$ to the class of the same order
for $kQ$ induces an isomorphism of posets
\[
\sss A\ \xrightarrow{\ \sim\ }\ \qhs(kQ).
\]
In particular, $\sss A$ depends only on $Q$, and not on the field $k$, the vertex
algebras or the free bimodules.
\item
The class of a reverse topological order is the maximum element of $\sss A$.
Its standard modules are $P(i)$, and its characteristic tilting
module is $A$.
\item
Let $\le$ be a topological order on $Q_0$, and let $\Delta$
denote the standard modules defined by $\le$.
Then $[\le]$ is the minimum element of $\sss A$,
$\Delta(i)=R_i$ for every $i$, and $\F(\Delta)$
is the category of locally free $A$-modules.
If every $R_i$ is self-injective, then $T_\le\cong\DD A$.
\item $\qhs A\ne\emptyset$ if and only if every $R_i$ is a division ring. In this case, $A$ is hereditary and $\sss A=\qhs A$.
\end{enumerate}
\end{theorem}

\begin{proof}
For a total order $\le$ on $Q_0$, we write $\Delta(i)$, $K(i)$ and $D_i$ for the corresponding standard modules, kernels and local algebras.

\textup{(1)}
Put $f_i=\sum_{j>i}e_j$. 
Then $K(i)=e_iAf_iA$ and $\Delta(i)=P(i)/K(i)$.
By the description of $A$ above, $K(i)$ is the direct sum of the $B_u$ over all paths $u=(i=v_0\to\cdots\to v_m)$ such that $v_\ell>i$ for some $\ell$.
Put
\[
\mathcal E_i=
\left\{
p=(i=v_0\to\cdots\to v_r)
\ \middle|\
r\ge1,v_\ell\le i\text{ for }0\le\ell<r, v_r>i
\right\}.
\]
For a path $u=(i=v_0\to\cdots\to v_m)$ occurring in $K(i)$,
let $r=\min\{\ell\mid v_\ell>i\}$.
Then $u=pq$, where $p=(i=v_0\to\cdots\to v_r)\in\mathcal E_i$ and $q=(v_r\to\cdots\to v_m)$.
This decomposition is unique, and $q$ may have length zero.
Conversely, every path obtained by extending
$p\in\mathcal E_i$ occurs in $K(i)$.
Since $P(v_r)=\bigoplus_q B_q$ as an $S$-bimodule, where $q$ runs over all paths starting at $v_r$, multiplication induces an isomorphism of right $A$-modules
\[
\bigoplus_{p\in\mathcal E_i}
B_p\otimes_{R_{v_r}}P(v_r)
\xrightarrow{\ \sim\ }K(i),
\]
where $v_r$ is the terminal vertex of $p$ in each summand.
Since $B_p$ is free as a right $R_{v_r}$-module,
each summand is a finite direct sum of copies of $P(v_r)$.
Thus $K(i)$ is projective and belongs to
$\add\bigl(\bigoplus_{j>i}P(j)\bigr)$.

We prove
$P(i)\in\F(\Delta(\ge i))$ by descending induction on $i$.
For the maximal vertex $i$, we have $K(i)=0$ and $P(i)=\Delta(i)$.
Assume the assertion holds for all $j>i$.
Then we have $K(i) \in \add\bigl(\bigoplus_{j>i}P(j)\bigr)\subseteq \F(\Delta(> i))$. 
By the exact sequence $0\to K(i)\to P(i)
\to\Delta(i)\to0$, we obtain $P(i)\in\F(\Delta(\ge i))$.
Hence $(A,\le)$ is standardly stratified.
Since $K(i)$ is projective, we have $\pd_A\Delta(i)\le1$.
Thus  $\pd M\le1$ for all $M\in\F(\Delta)$. 

\textup{(2)}
For $i,j\in Q_0$, put
\[
\mathcal P_{ij}^{\le}=
\left\{
p=(i=v_0\to\cdots\to v_r=j)
\ \middle|\
v_\ell\le i\text{ for }0\le\ell\le r
\right\},
\]
including the length zero path when $i=j$.
By the description of $K(i)$ in (1), we have
\[
\Delta(i)e_j
\cong
\bigoplus_{p\in\mathcal P_{ij}^{\le}}B_p
\]
as right $R_j$-modules.
Since every $B_p$ is nonzero and free as a right $R_j$-module, $\Delta(i)$ is locally free, and
$\supp\Delta(i)
=\{j\mid\mathcal P_{ij}^{\le}\ne\emptyset\}
=\Sigma_i^\le$. 
Since $Q$ is acyclic, the only path from $i$ to $i$ is the length zero path.
Thus $K(i)e_i=0$, and $D_i=e_iAe_i/K(i)e_i=R_i$.

\textup{(3)}
Let $\le$ and $\le'$ be total orders on $Q_0$. 
By (1) and (2), we have  $\pd_A\Delta^\le(i)\le1$ and $\supp\Delta^\le(i)=\Sigma_i^\le$ for every $i$. 
Moreover, since $D_i^\le=R_i=D_i^{\le'}$ by \textup{(2)}, Lemma~\ref{lem:Delta}(1) gives $[\Delta^\le(i):S(i)]=[\Delta^{\le'}(i):S(i)]$ for every $i$.
 Hence Proposition~\ref{prop:support-order} shows that $[\le]\preceq[\le']$ if and only if $\Sigma_i^\le\subseteq\Sigma_i^{\le'}$ for every $i$.
Thus the first assertion holds.

The path algebra $kQ$ is the acyclic tensor algebra with $R_i=k$ and $B_a=k$, and every total order defines a quasi-hereditary structure on $kQ$ by (1) and (2).
For $E=A$ or $kQ$, write $[\le]_E$ for the class of a total order $\le$ in $\sss E$.
By the first assertion for $A$ and $kQ$, both $[\le]_A\preceq[\le']_A$ and $[\le]_{kQ}\preceq[\le']_{kQ}$ are equivalent to $\Sigma_i^\le\subseteq\Sigma_i^{\le'}$ for every $i\in Q_0$.
Since every class in $\sss A$ and in $\qhs(kQ)$ contains a total order by Lemma~\ref{lem:refine}, the assignment $[\le]_A\mapsto[\le]_{kQ}$ is a well-defined isomorphism of posets $\sss A\to\qhs(kQ)$.
The last assertion holds since the set $\Sigma_i^\le$ depends only on $Q$ and $\le$.

\textup{(4)}
Since $e_iAe_j\ne0$ if and only if there is a path from $i$ to $j$ in $Q$, the relation
$\preceq_A$ of Corollary~\ref{cor:regular} satisfies $j\preceq_A i$ if and only if
there is a path from $i$ to $j$. Since $Q$ is acyclic, $\preceq_A$ is a partial order,
and every reverse topological order $\le$ is a linear extension of $\preceq_A$. Hence
the assertion follows from Lemma~\ref{lem:refine} and Corollary~\ref{cor:regular}.

\textup{(5)}
Let $\le$ be a topological order on $Q_0$. 
For every path $p=(i=v_0\to v_1\to\cdots\to v_r)$ of positive length, we have $i=v_0<v_1<\cdots<v_r$.
Thus only the length zero path at $i$ contributes to the formula for $\Delta(i)e_j$ in (2).
Since $\Delta(i)e_i\cong R_i$ and $\Delta(i)e_j=0$ for $j\ne i$, we have $\Delta(i)=R_i$, regarded as an $A$-module concentrated at vertex $i$.

We show that $\F(\Delta)$ coincides with the subcategory of locally free modules.
Let
$0\to X\to Y\to Z\to0$
be an exact sequence of $A$-modules with $X$ and $Z$ locally free.
For every vertex $j$, the exact sequence
$0\to Xe_j\to Ye_j \to Ze_j\to0$
splits as a sequence of right $R_j$-modules,
since $Ze_j$ is free.
Hence $Ye_j\cong Xe_j\oplus Ze_j$ is free.
Thus the subcategory of locally free modules is closed under extensions.
Since each $\Delta(i)=R_i$ is locally free,
every module in $\F(\Delta)$ is locally free.

Conversely, let $M$ be locally free.
We prove $M\in\F(\Delta)$ by induction on $|\supp M|$.
The assertion is clear for $M=0$.
Suppose $M\ne0$, and let $j$ be the maximal vertex in $\supp M$.
Since $\le$ is topological, $e_jAe_\ell=0$ for $\ell<j$.
By the choice of $j$, we also have $Me_\ell=0$ for $\ell>j$.
Hence $N=Me_j$ is an $A$-submodule of $M$.
Moreover, since $Ne_j=Me_j$ and $Ne_\ell=0$ for $\ell\ne j$, the module $N$ is concentrated at vertex $j$.
Since $M$ is locally free, $Me_j\cong R_j^{\oplus r_j}$ as right $R_j$-modules for some $r_j$.
Since $\Delta(j)=R_j$ is also concentrated at vertex $j$, we obtain $N\cong\Delta(j)^{\oplus r_j}$ as $A$-modules.
Moreover, 
\[
(M/N)e_\ell\cong Me_\ell/Ne_\ell\cong
\begin{cases}
0,&\ell=j,\\
Me_\ell,&\ell\ne j.
\end{cases}
\]
Thus $M/N$ is locally free and $\supp(M/N)=\supp M\setminus\{j\}$.
By the induction hypothesis, $M/N\in\F(\Delta)$.
Since $N\cong\Delta(j)^{\oplus r_j}\in\F(\Delta)$, we have $M\in\F(\Delta)$. 

We show that $[\le]$ is the minimum element of $\sss A$.
By (2) and Lemma~\ref{lem:refine}, every standard module $\Delta_s(i)$, for $s\in\sss A$, is locally free.
Since the subcategory of locally free modules is closed under extensions,
we obtain $\F(\Delta_s)\subseteq\F(\Delta)$.
Hence $[\le]\preceq s$ for every $s\in\sss A$,
which proves that $[\le]$ is the minimum of $\sss A$. 

Finally, suppose that every $R_j$ is self-injective.
Since $R_j$ is local and self-injective,
both $R_j$ and $\DD R_j$ are indecomposable injective
right $R_j$-modules, and hence $\DD R_j\cong R_j$ as right $R_j$-modules.
By the decomposition of $A$ as an $S$-bimodule, $e_jA$ is free as a left $R_j$-module.
Writing $e_jA\cong R_j^{\oplus m_j}$ as left $R_j$-modules,
we obtain $(\DD A)e_j\cong\DD(e_jA)\cong(\DD R_j)^{\oplus m_j}\cong R_j^{\oplus m_j}$ as right $R_j$-modules.
Thus $\DD A$ is locally free and belongs to $\F(\Delta)$.
Since $\DD A$ is injective, it also belongs to $\Y(\Delta)$.
By Theorem~\ref{thm:chtilt}(1), we have
\[
\DD A\in\F(\Delta)\cap\Y(\Delta)=\add T_\le.
\]
Since both $\DD A$ and $T_\le$ are basic with $|Q_0|$
indecomposable summands, we obtain $T_\le\cong\DD A$.

\textup{(6)}
For $[\unlhd]\in\sss A$ and a linear extension $\le$ of $\unlhd$, (2) and Lemma~\ref{lem:refine} give $D_i^\unlhd=D_i^{\le}=R_i$ for every $i\in Q_0$.
Hence a structure in $\sss A$ is quasi-hereditary if and only if every $R_i$ is a division ring. 
Since $\sss A\ne\emptyset$ by (1), the first assertion follows, and $\sss A=\qhs A$ in this case.
Moreover, $S$ is semisimple, and hence $A=T_S(B)$ is hereditary, since every right $A$-module $M$ has the projective resolution $0\to M\otimes_SB\otimes_SA\to M\otimes_SA\to M\to0$.
\end{proof}

\begin{remark}\label{rem:Li-tensor}
Over an algebraically closed field, Theorem~\ref{thm:tensor}(1) also follows from \cite[Theorem~0.1]{Li}.
Our proof works over an arbitrary field and gives an explicit description of the standard modules.
For $A=k[x]/(x^m)\otimes_k kQ$, the first assertion of (1),
statement (2) and the description of $T_\le$ in (5)
also follow from \cite[Proposition~35 and Corollary~37]{FM}. 
\end{remark}

By Theorems~\ref{thm:bijection} and \ref{thm:tensor}(3), we describe the poset of triangular tilting modules in terms of the quiver $Q$.

\begin{corollary}\label{cor:tensor-Ttilt}
Let $A=T_S(B)$ be an acyclic tensor algebra with quiver $Q$. Then there are isomorphisms
of posets
\[
\Ttilt A\cong\sss A\cong\qhs(kQ)\cong\UTtilt(kQ).
\]
In particular, the poset $\Ttilt A$ depends only on $Q$.
\end{corollary}

\begin{proof}
This follows from Theorems~\ref{thm:bijection} and~\ref{thm:tensor}(3).
\end{proof}

Combining Theorem~\ref{thm:tensor}(3) with the results of \cite{FKR}, we describe $\sss A$ and compute its cardinality when the underlying graph of $Q$ is of type $\mathbb{A}_n$.
Let $\operatorname{Tam}_n$ denote the Tamari lattice on full binary trees with $n$ internal vertices, and put
$\operatorname{Cat}_n=\frac{1}{n+1}\binom{2n}{n}.$

\begin{corollary}\label{cor:GLS-Catalan}
Let $A=T_S(B)$ be an acyclic tensor algebra with quiver $Q$, and suppose the underlying graph of $Q$ is of type $\mathbb{A}_n$. 
Let $Q_1,\ldots,Q_r$ be the maximal connected full subquivers of $Q$ whose arrows all point in the same direction along the underlying path.
Put $n_a=|(Q_a)_0|$.
Then
\[
\sss A\cong\prod_{a=1}^r\operatorname{Tam}_{n_a},
\qquad
|\sss A|=\prod_{a=1}^r\operatorname{Cat}_{n_a}.
\]
In particular, if $Q$ is equioriented with $n$ vertices, then $\sss A\cong\operatorname{Tam}_n$ and $|\sss A|=\operatorname{Cat}_n$.
\end{corollary}

\begin{proof}
By Theorem~\ref{thm:tensor}(3), we have an isomorphism of posets $\sss A\cong\qhs(kQ)$.
Since every vertex at which the orientation changes
is a source or a sink of $Q$, the subquivers
$Q_1,\ldots,Q_r$ are obtained by decomposing $Q$
at sources and sinks.
Since $(kQ)/\langle e_i\mid i\notin(Q_a)_0\rangle\cong kQ_a$ for each $a$,  there is an isomorphism of posets
$\qhs(kQ)\cong\prod_{a=1}^r\qhs(kQ_a)$ by  \cite[Theorem~3.7]{FKR}. 
Since each $Q_a$ is an equioriented quiver of type
$\mathbb{A}_{n_a}$, it follows from \cite[Corollary~4.8]{FKR} that 
$\qhs(kQ_a)\cong\operatorname{Tam}_{n_a}$ and 
$|\qhs(kQ_a)|=\operatorname{Cat}_{n_a}$. 
Hence we have
\[
\sss A\cong\qhs(kQ)\cong\prod_{a=1}^r\qhs(kQ_a)\cong\prod_{a=1}^r\operatorname{Tam}_{n_a}, \quad |\sss A|=\prod_{a=1}^r\operatorname{Cat}_{n_a}.\]
If $Q$ is equioriented, then $r=1$ and $Q_1=Q$, and the last assertion follows.
\end{proof}

Next, we apply Theorem~\ref{thm:tensor} to GLS path algebras.

Let $k$ be an arbitrary field, let $C$ be a symmetrizable generalized Cartan matrix, let
$D=\operatorname{diag}(d_i)_{i\in I}$ be a symmetrizer, and let $\Omega$ be an acyclic orientation. 
Let $H=H_k(C,D,\Omega)$ be the algebra defined in \cite{GLS-found}.
We call such algebras \emph{GLS path algebras}. The vertex algebras of $H$ are
\[
H_i=e_iHe_i\cong k[\epsilon_i]/(\epsilon_i^{d_i}).
\]
The algebra $H$ has a tensor algebra presentation
$H\cong T_S(B)$, where $S=\prod_{i\in I}H_i$, and each nonzero $H_i$--$H_j$-bimodule $e_iBe_j$
is free as a left $H_i$-module and as a right $H_j$-module
\cite[Section~5.1 and Proposition~6.4]{GLS-found}.
Let $Q^\circ$ be the acyclic quiver with vertex set $I$
and one arrow $i\to j$ whenever $e_iBe_j\ne0$.

By Theorem~\ref{thm:tensor}, we have the following result. 

\begin{corollary}\label{cor:GLS}
Let $H=H_k(C,D,\Omega)$ be a GLS path algebra.
Then the following statements hold. 
\begin{enumerate}[label=\textup{(\arabic*)}]
\item $H$ is standardly stratified with respect to every total order on $I$. For total orders $\le$ on $I$, the assignment $[\le]_H\mapsto[\le]_{kQ^\circ}$ defines an isomorphism of posets $\sss H\rightarrow\qhs(kQ^\circ)$.
\item The poset $\sss H$ has a maximum element $s_{\max}$ and a minimum element $s_{\min}$.
Their characteristic tilting modules are $H$ and $\DD H$, respectively.
Moreover, $\F(\Delta_{s_{\min}})$ coincides with the subcategory of locally free $H$-modules. 
\item $\qhs H\ne\emptyset$ if and only if $d_i=1$ for every $i$. In this case, $H$ is a path algebra, possibly with parallel arrows, and $\sss H=\qhs H$.
\end{enumerate}
\end{corollary}

\begin{proof}
The vertex algebras $H_i$ are local and self-injective,
and $H_i$ is a division ring if and only if $d_i=1$.
Thus (1) follows from Theorem~\ref{thm:tensor}(1),(3), and (2) follows from Theorem~\ref{thm:tensor}(4),(5).
We show (3). 
By Theorem~\ref{thm:tensor}(6), $\qhs H\ne\emptyset$ if and only if $d_i=1$ for every $i\in I$.
In this case, $\sss H=\qhs H$.
If $d_i=1$ for every $i$, then $H_i=k$ for every $i$.
Choosing a $k$-basis of each nonzero $e_iBe_j$, the tensor algebra presentation identifies $H$ with the path algebra of an acyclic quiver, with one arrow for each chosen basis element.
\end{proof}

By Corollary~\ref{cor:tensor-Ttilt}, the poset $\Ttilt H$ depends only on $Q^\circ$.

\begin{remark}\label{rem:GLS-diagonal}
Let $T=\bigoplus_{i\in I}T(i)$ be a labelled decomposition of $T\in\Ttilt H$ that is triangular with respect to $\unlhd$. Since $D^\unlhd_i=H_i$ by Theorem~\ref{thm:tensor}(2) and Lemma~\ref{lem:refine}, Theorem~\ref{thm:main}(b) gives $T(i)e_i\cong H_i$ as right $H_i$-modules. In particular, $[T(i):S(i)]=d_i$ for every $i$, and $T$ is not unitriangular unless $d_i=1$ for every $i$.
\end{remark}

We give an example of Corollary~\ref{cor:GLS-Catalan}. 

\begin{example}\label{ex:GLS-counts}
If $Q^\circ$ is an equioriented quiver of type
$\mathbb{A}_n$, then Corollary~\ref{cor:GLS-Catalan}
gives $|\sss H|=\operatorname{Cat}_n$.
For
\[
Q^\circ:\quad
1\longrightarrow2\longleftarrow3
\longleftarrow4\longrightarrow5,
\]
the subquivers in Corollary~\ref{cor:GLS-Catalan} are
\[
Q_1:\quad1\longrightarrow2,\qquad
Q_2:\quad2\longleftarrow3\longleftarrow4,\qquad
Q_3:\quad4\longrightarrow5.
\]
Thus $n_1=2$, $n_2=3$, and $n_3=2$, and hence
$
|\sss H|
=\operatorname{Cat}_2\operatorname{Cat}_3
 \operatorname{Cat}_2
=20.
$

For Dynkin types $\mathbb{B}_n$ and $\mathbb{C}_n$, the underlying graph of $Q^\circ$ is of type $\mathbb{A}_n$.
For types $\mathbb{F}_4$ and $\mathbb{G}_2$, it is of type $\mathbb{A}_4$ and $\mathbb{A}_2$, respectively.
Thus if $C$ is of Dynkin type $\mathbb{B}_n$ or $\mathbb{C}_n$ and $Q^\circ$ is equioriented, then Corollary~\ref{cor:GLS-Catalan} gives $|\sss H|=\operatorname{Cat}_n$.
For Dynkin types $\mathbb{F}_4$ and $\mathbb{G}_2$ with $Q^\circ$ equioriented, the numbers are $\operatorname{Cat}_4=14$ and $\operatorname{Cat}_2=2$, respectively.
\end{example}

\subsection{Triangular tilting modules under the GLS correspondence}\label{subsec:generic}
Using Theorem~\ref{thm:main}, we show that the tilting
correspondence of \cite{GLS-rigid} preserves
characteristic tilting modules.
Throughout this subsection, $H=H_k(C,D,\Omega)$ is a GLS path algebra as in Section~\ref{subsec:support}.

Put $\mathcal O=k[[t]]$ and $K=k((t))$.
Let $\widehat H$ and $\widetilde H$ be the $\mathcal O$-order and the finite dimensional
hereditary $K$-algebra constructed in \cite[Sections~4.1 and~4.2]{GLS-rigid}.
By \cite[Proposition~4.1(c),(d)]{GLS-rigid}, we have
\[
\widehat H/t\widehat H\cong H,\qquad K\otimes_{\mathcal O}\widehat H\cong\widetilde H.
\]
The algebra $\widetilde H$ is an acyclic tensor algebra with vertex fields
$K_i=e_i\widetilde He_i$ and quiver $Q^\circ$ (see \cite[Section~4.1]{GLS-rigid}).
In general, $\widetilde H$ is not isomorphic to the path algebra $KQ^\circ$.

For $E=H$ or $\widehat H$, let $\mod_{\mathrm{lf}}E$ denote the category of locally free
right $E$-modules. We consider the functors
\[\operatorname{Red}=-\otimes_{\widehat H}H\colon \mod_{\mathrm{lf}}\widehat H\longrightarrow \mod_{\mathrm{lf}}H,\qquad\operatorname{Loc}=-\otimes_{\widehat H}\widetilde H\colon\mod_{\mathrm{lf}}\widehat H
\longrightarrow\mod\widetilde H. \]
By \cite[Theorem~1.1]{GLS-rigid}, these functors induce bijections between the isomorphism classes of rigid modules in the three categories, and these bijections and their inverses preserve indecomposability and tilting modules. 
Combining these bijections, we obtain a bijection $\beta$ between
the isomorphism classes of rigid locally free $H$-modules and rigid
$\widetilde H$-modules, given by
$\beta(\operatorname{Red}(\widehat M))
\cong\operatorname{Loc}(\widehat M)$
for every rigid locally free $\widehat H$-module $\widehat M$.
Since $\operatorname{Red}$ and $\operatorname{Loc}$ are additive and these bijections
preserve indecomposability, $\beta$ and $\beta^{-1}$ also preserve direct sums.
Since every tilting $H$-module is locally free and has projective dimension at most one
by \cite[Theorem~1.2]{GLS-found} (see also \cite[Proposition~3.2(a)]{GLS-rigid}), $\beta$ restricts to a bijection
\[
\beta\colon\tilt H\longrightarrow\tilt\widetilde H.
\]
Moreover, by \cite[Lemma~5.1]{GLS-rigid}, we have
$\operatorname{rank}_{H_i}(Me_i)=\dim_{K_i}(\beta(M)e_i)$ for every rigid locally free
$H$-module $M$ and every $i\in I$. Hence
\begin{equation}\label{eq:supp-pres}
\supp_H M=\supp_{\widetilde H}\beta(M).
\end{equation}
This support equality allows us to apply Theorem~\ref{thm:main} to compare the
characteristic tilting modules of $H$ and $\widetilde H$.

\begin{theorem}\label{thm:GLS-compatibility}
Let $H$ be a GLS path algebra. Then the following statements hold.
\begin{enumerate}[label=\textup{(\arabic*)}]
\item
For every partial order $\unlhd$ on $I$,
$(H,\unlhd)$ is standardly stratified if and only if
$(\widetilde H,\unlhd)$ is quasi-hereditary.
In this case, $\beta(T^H_\unlhd(i))
\cong T^{\widetilde H}_\unlhd(i)$ for every $i\in I$.

\item
Let $T\in\Ttilt H$ with a triangular labelled decomposition
$T=\bigoplus_{i\in I}T(i)$.
Then $\beta(T)\cong\bigoplus_{i\in I}\beta(T(i))$
is a triangular labelled decomposition of $\beta(T)$.
Moreover, ${\preceq_{\beta(T)}}={\preceq_T}$.

\item
The following diagram commutes, and all its maps
are isomorphisms of posets.
\[
\begin{array}{ccc}
\sss H
&\xrightarrow{\ \Phi\ }&
\qhs\widetilde H\\[2mm]
\big\downarrow\scriptstyle\psi_H
&&
\big\downarrow\scriptstyle\psi_{\widetilde H}\\[2mm]
\Ttilt H
&\xrightarrow{\ \beta\ }&
\UTtilt\widetilde H,
\end{array}
\]
where $\Phi([\unlhd]_H)=[\unlhd]_{\widetilde H}$,
and the vertical maps are those of
Theorem~\ref{thm:bijection}.
\end{enumerate}
\end{theorem}

\begin{proof}
Since the vertex algebras of $\widetilde H$ are fields, $\sss\widetilde H=\qhs\widetilde H$ by Theorem~\ref{thm:tensor}(6).

\textup{(1)}
Let $(E,E',\gamma)$ be $(H,\widetilde H,\beta)$ or $(\widetilde H,H,\beta^{-1})$, and
suppose that $(E,\unlhd)$ is standardly stratified. Then $\gamma(T^E_\unlhd)$ is a
tilting $E'$-module with labelled decomposition
$\bigoplus_{i\in I}\gamma(T^E_\unlhd(i))$. By Theorem~\ref{thm:main} and
\eqref{eq:supp-pres},
\[
\supp_{E'}\gamma(T^E_\unlhd(i))=\supp_ET^E_\unlhd(i)\subseteq\{j\in I\mid j\unlhd i\}
\]
for each $i\in I$. By Theorem~\ref{thm:main}, $(E',\unlhd)$ is standardly stratified
and $\gamma(T^E_\unlhd(i))\cong T^{E'}_\unlhd(i)$. Since $\sss\widetilde H=\qhs\widetilde H$,
the assertion follows.

\textup{(2)}
Since $\beta$ preserves direct sums and indecomposability, $\beta(T)\cong\bigoplus_{i\in I}\beta(T(i))$ is a labelled decomposition of the basic
tilting $\widetilde H$-module $\beta(T)$. By \eqref{eq:supp-pres}, we have $\supp_{\widetilde H}\beta(T(i))=\supp_HT(i)$ for every $i$. 
Hence this labelled decomposition is triangular. 
Moreover, ${\preceq_{\beta(T)}}={\preceq_T}$ by Definition~\ref{def:canonical}(2).

\textup{(3)}
By Lemma~\ref{lem:refine}, every class in $\sss H$ and in $\sss\widetilde H=\qhs\widetilde H$ contains a total order. 
Since $H$ and $\widetilde H$ are acyclic tensor algebras with the same quiver $Q^\circ$, Theorem~\ref{thm:tensor}(1) and the first assertion of Theorem~\ref{thm:tensor}(3), applied to $H$ and $\widetilde H$, show that the assignment $[\le]_H\mapsto[\le]_{\widetilde H}$ for total orders $\le$ on $I$ is a well-defined isomorphism of posets $\sss H\to\qhs\widetilde H$.
If $(H,\unlhd)$ is standardly stratified and $\le$ is a linear extension of $\unlhd$, then $[\unlhd]_H=[\le]_H$ and $[\unlhd]_{\widetilde H}=[\le]_{\widetilde H}$ by \textup{(1)} and Lemma~\ref{lem:refine}. 
Hence this isomorphism coincides with $\Phi$.

For $s\in\sss H$, we have $\beta(\psi_H(s))=\psi_{\widetilde H}(\Phi(s))$ by
\textup{(1)}. Thus the diagram commutes, and
$\left.\beta\right|_{\Ttilt H}=\psi_{\widetilde H}\circ\Phi\circ\psi_H^{-1}$.
Since $\Phi$, $\psi_H$ and $\psi_{\widetilde H}$ are isomorphisms of posets by the argument 
above and Theorem~\ref{thm:bijection}, the restriction of $\beta$ is also an
isomorphism of posets.
\end{proof}

We use the GLS correspondence to compare $\Ttilt A$ and $\tilt A$ for some GLS path algebras.
For the examples and corollary below, put $R^{(m)}=k[x]/(x^m)$, where $m\ge1$.
For a finite acyclic quiver $Q$, let $r_{ij}$ be the number of arrows between vertices $i$ and $j$.
Let $C_Q=(c_{ij})$ be the symmetric generalized Cartan matrix defined by $c_{ii}=2$ and $c_{ij}=-r_{ij}$ if $i\ne j$. 
Put $\Omega_Q=\{(i,j)\in Q_0\times Q_0\mid Q\text{ has an arrow }i\to j\}.$
By the tensor algebra presentation, we have $R^{(m)}\otimes_k kQ\cong H_k(C_Q,mI_{|Q_0|},\Omega_Q)$.
By the construction in \cite[Sections~4.1--4.2]{GLS-rigid}, we have $\widehat H\cong k[[u]]Q$, $\widetilde H\cong k((u))Q$ and $t=u^m$.
We denote the corresponding bijection by $\beta_{m,Q}\colon\tilt(R^{(m)}\otimes_k kQ)\longrightarrow\tilt(k((u))Q)$.

The following examples show that $\Ttilt A$ is a proper subset of $\tilt A$ in general.

\begin{example}\label{ex:GLS-fork}
Let $Q:1\to2\leftarrow3$, and put
$A_m=R^{(m)}\otimes_k kQ$.
By Theorem~\ref{thm:bijection}
and Corollary~\ref{cor:GLS-Catalan}, we have
\[
\sss A_m\cong\operatorname{Tam}_2
\times\operatorname{Tam}_2,
\qquad
|\sss A_m|=|\Ttilt A_m|=4.
\]
Put $\widetilde A_m=k((u))Q$.
The basic tilting $\widetilde A_m$-modules are precisely the following.
\[
\begin{gathered}
P(1)\oplus P(2)\oplus P(3),\qquad
P(1)\oplus P(3)\oplus I(2),\\
P(1)\oplus I(2)\oplus S(1),\qquad
P(3)\oplus I(2)\oplus S(3),\qquad
I(2)\oplus S(1)\oplus S(3).
\end{gathered}
\]
By the bijection $\beta_{m,Q}$, we have $|\tilt A_m|=5$.

Let $U=P(1)\oplus P(3)\oplus I(2)$.
Its indecomposable summands have supports
$\{1,2\}$, $\{2,3\}$ and $\{1,2,3\}$.
By Corollary~\ref{cor:algorithm}, every triangular tilting module has an
indecomposable direct summand supported at a single vertex. Thus $U$ is not triangular.
By Theorem~\ref{thm:GLS-compatibility},
$\beta_{m,Q}^{-1}(U)$ is not triangular either.
Hence $\Ttilt A_m\subsetneq\tilt A_m$.
For $m\ge2$, we have
$\qhs A_m=\emptyset$ by Corollary~\ref{cor:GLS}(3).
\end{example}

\begin{example}\label{ex:GLS-Kronecker}
Let $Q$ be the Kronecker quiver
$1\rightrightarrows2$, and put
$A_m=R^{(m)}\otimes_k kQ$.
As a tensor algebra, $A_m$ has vertex algebras $R^{(m)}$
and quiver $Q^\circ:1\to2$.
The $R^{(m)}$--$R^{(m)}$-bimodule corresponding to this arrow
is $(R^{(m)})^{\oplus2}$.
Thus Theorems~\ref{thm:bijection}
and~\ref{thm:tensor}(3) give
\[
|\sss A_m|=|\Ttilt A_m|=2.
\]
On the other hand, the Kronecker algebra $k((u))Q$
has infinitely many basic tilting modules.
Thus the bijection $\beta_{m,Q}$ gives
$|\tilt A_m|=\infty$.
\end{example}

The following corollary gives algebras for which
every basic tilting module is triangular.

\begin{corollary}\label{cor:GLS-linear-tilting}
Let $m,n\ge1$, and put
\[
A_{m,n}=R^{(m)}\otimes_k k\vec{\mathbb{A}}_n,
\qquad
\vec{\mathbb{A}}_n:1\longrightarrow2\longrightarrow\cdots
\longrightarrow n.
\]
Then there are isomorphisms of posets
\[
\sss A_{m,n}
\cong\Ttilt A_{m,n}
=\tilt A_{m,n}
\cong\operatorname{Tam}_n.
\]
In particular,
$|\tilt A_{m,n}|=\operatorname{Cat}_n$.
If $m\ge2$, then $\qhs A_{m,n}=\emptyset$.
\end{corollary}

\begin{proof}
By \cite[Theorem~4.7]{FKR} and Theorem~\ref{thm:bijection}, we have
$\tilt(k((u))\vec{\mathbb{A}}_n)=\UTtilt(k((u))\vec{\mathbb{A}}_n)$.
Since the bijection $\beta_{m,\vec{\mathbb{A}}_n}$ maps $\Ttilt A_{m,n}$ onto
$\UTtilt(k((u))\vec{\mathbb{A}}_n)$ by Theorem~\ref{thm:GLS-compatibility}(3),
we obtain $\tilt A_{m,n}=\Ttilt A_{m,n}$.
The poset isomorphisms and counting formula follow
from Theorem~\ref{thm:bijection}
and Corollary~\ref{cor:GLS-Catalan}.
If $m\ge2$, Corollary~\ref{cor:GLS}(3) gives
$\qhs A_{m,n}=\emptyset$.
\end{proof}

Corollary~\ref{cor:GLS-linear-tilting} does not extend to non-symmetric Cartan matrices, even if $Q^\circ$ has only two vertices.

\begin{example}\label{ex:GLS-B2}
Let $H=H_k(C,D,\Omega)$ with $C=\begin{pmatrix}2&-1\\-2&2\end{pmatrix}$, $D=\begin{pmatrix}2&0\\0&1\end{pmatrix}$ and $\Omega=\{(1,2)\}$. 
Then $C$ is of Dynkin type $\mathbb{B}_2$ and $Q^\circ\colon1\to2$.
In this case, $H\cong kQ/I$, where
\[
Q:\quad
\xymatrix{
1\ar@(ul,dl)[]_{\alpha}\ar[r]^{a}&2
},
\qquad
I=\langle\alpha^2\rangle.
\]

Every element of $\sss H$ contains one of the two total orders on $\{1,2\}$ by
Lemma~\ref{lem:refine}. Hence $\Ttilt H=\{H,\DD H\}$ by Theorem~\ref{thm:bijection}
and Corollary~\ref{cor:GLS}(2).
On the other hand, put $T=P(1)\oplus I(2)$. 
Then $T$ is a tilting $H$-module.
Since $\supp P(1)=\supp I(2)=\{1,2\}$, the module $T$ is not triangular by
Corollary~\ref{cor:algorithm}. Hence $\Ttilt H\subsetneq\tilt H$.
\end{example}

If $k$ is algebraically closed and $A$ is connected
with $\qhs A\ne\emptyset$, then $\tilt A=\UTtilt A$ holds if and only if
$A$ is a quadratic linear Nakayama algebra by \cite[Theorem~4.1]{ACKT}. 

For $m\ge2$, Corollary~\ref{cor:GLS-linear-tilting}
gives algebras $A_{m,n}$ with
$\qhs A_{m,n}=\emptyset$ and
$\tilt A_{m,n}=\Ttilt A_{m,n}$.
This leads to the following question.

\begin{question}\label{q:classes}
For which basic finite dimensional algebras $A$
is every basic tilting module triangular,
that is, $\tilt A=\Ttilt A$?
\end{question}

\end{document}